\documentclass[a4paper]{article}

\usepackage[utf8]{inputenc}
\usepackage[T1]{fontenc}

\usepackage{amstext}
\usepackage{amsfonts}
\usepackage{amsmath}
\usepackage{amssymb}
\usepackage{amsthm}
\usepackage{url}

\newtheorem{theorem}{Theorem}
\newtheorem{proposition}[theorem]{Proposition}
\newtheorem{lemma}[theorem]{Lemma}
\newtheorem{definition}[theorem]{Definition}

\newtheorem{corollary}[theorem]{Corollary}
\newtheorem{question}{Question}
\newtheorem{questions}[question]{Questions}
\newtheorem*{theoremA}{Theorem A}
\newtheorem*{theoremB}{Theorem B}
\newtheorem*{theoremC}{Theorem C}
\newtheorem*{theoremD}{Theorem D}
\newtheorem*{remark}{Remark}

\newcommand{\Dim}{\mathrm{dim}}
\newcommand{\DimT}{\underline{\mathrm{dim}}^{\mathrm{T}}}
\newcommand{\Mrad}{{\mathcal{M}}_{\mathrm{r}}^{\mathrm{1}}}
\newcommand{\R}{\mathbf{R}}
\newcommand{\diminf}{\underline{\dim}}
\newcommand{\dimsup}{\overline{\dim}}
\newcommand{\BMS}{m_{\mathrm{BMS}}}

\newcommand{\Hc}{\mathbf{H}_{\mathbf{C}}}
\newcommand{\C}{\mathbf{C}}
\newcommand{\HH}{\mathcal{H}}
\title{Hausdorff dimension of limit sets}
\author{Laurent Dufloux}

\begin{document}
\maketitle
\abstract{We exhibit a class of Schottky subgroups of $\mathbf{PU}(1,n)$ ($n \geq 2$) which we call \emph{well-positioned} and show that the Hausdorff dimension of the limit set $\Lambda_\Gamma$ associated with such a subgroup $\Gamma$, with respect to the spherical metric on the boundary of complex hyperbolic $n$-space, is equal to the growth exponent $\delta_\Gamma$. 

For general $\Gamma$ we establish (under rather mild hypotheses) a lower bound involving the dimension of the Patterson-Sullivan measure along boundaries of complex geodesics.

Our main tool is a version of the celebrated Ledrappier-Young theorem.}

\section{Introduction}
Let $\Gamma$ be a discrete non-elementary subgroup of $\mathbf{PU}(1,n)$, and let $\Lambda_\Gamma$ be its limit set; this is a subset of the boundary of the complex hyperbolic $n$-space, $\partial \Hc^n$. According to a theorem of Bishop and Jones (\cite{BishopJones}, see also \cite{Stratmann}), the Hausdorff dimension of the conical limit set $\Lambda_\Gamma^c$, with respect to a Gromov metric on the boundary, is equal to the Growth exponent $\delta_\Gamma$. According to Bishop and Jones, the proof of this theorem ``uses nothing but the definitions and a few simple properties of M\"{o}bius transformations''. To be precise, Bishop and Jones deal only with the real hyperbolic space, but their theorem was generalized further to the case of pinched curvature by Paulin \cite{PaulinBJ}.

The reader should bear in mind that the operation of $\mathbf{PU}(1,n)$ on the boundary is \emph{conformal} with respect to a given Gromov metric. Conformality usually makes questions related to dimension of sets or measures more tractable than their non-conformal counterparts. In the presence of non-conformality, various phenomena may arise, we refer the reader to \cite{ChenPesin}.

In this paper, we wish to demonstrate the usefulness of the celebrated Ledrappier-Young formula in studying Hausdorff dimension of limit sets in non-conformal setting. More precisely, we are interested in estimating the Hausdorff dimension of $\Lambda_\Gamma^c$ with respect to a \emph{spherical} metric on the boundary. 

In fact looking at the boundary $\partial \Hc^n$ (minus some point) endowed with the Gromov metric (resp.  the spherical metric) amounts to looking at the classical Heisenberg space of height $n-1$, $\HH_{n-1}$ (resp. the Euclidean space $\R^{2n-1}$). The Heisenberg space $\HH_{n-1}$ has Hausdorff dimension $2n$ whereas Euclidean space $\R^{2n-1}$ has Hausdorff dimension $2n-1$. The main question we are interested in is thus a special case of the more general ``Gromov problem'' of relating the (Hausdorff) dimension of some given subset $A$ of the Heisenberg space $\HH$ endowed with the Heisenberg metric, $\mathrm{dim}_{\HH}(A)$, to the dimension $\mathrm{dim}_E(A)$ of $A$ with respect to the Euclidean metric.

This general problem has been worked out by Balogh \emph{et al.} in \cite{BaloghTyson}. Let us state their result now (see also theorem \ref{th.balogh} below). If we denote by $\delta$ the dimension of $A$ with respect to the Heisenberg metric, then the following sharp inequalities hold:
\[ \sup \left\{ \delta-1,\frac{\delta}{2} \right\} \leq \mathrm{dim}_E(A) \leq \inf \left\{  \delta,n-1+\frac{\delta}{2} \right\} \text. \]
Here, sharpness means that these inequalities cannot be improved without further assumptions. In this paper we are going to improve on the lower bound -- this is usually the difficult part in dealing with Hausdorff dimension -- when $A$ is some special kind of fractal set, namely the limit set of a discrete subgroup $\Gamma$ of $G$ satisfying some mild hypotheses.

We will assume that $\Gamma$ is Zariski-dense and has finite Bowen-Margulis-Sullivan measure. I should emphasize that I am not able to prove an exact formula in this general setting. Indeed, we obtain the following
\begin{theoremA}[theorem \ref{th.lower.bound}]
Let $\Gamma$ be a non-elementary discrete subgroup of $\mathbf{PU}(1,n)$, Zariski-dense, with finite BMS measure. If $\mu$ is some Patterson-Sullivan measure of exponent $\delta_\Gamma$, then for $\mu$-almost every $\xi$, 
\[ \delta_\Gamma - \frac{1}{2} \Dim(\lambda,Z)  \leq \diminf(\mu,\xi) \]
where the lower pointwise dimension is with respect to the \emph{spherical} metric on the boundary.

The same inequality holds if we replace $\diminf(\mu,\xi)$ with $\mathrm{dim}_E(\Lambda_\Gamma^c)$.
\end{theoremA}
Here, $\diminf(\mu,\xi)$ is the lower pointwise dimension of $\mu$ at $\xi$, \emph{i.e.}
\[ \diminf(\mu,\xi) = \liminf_{\rho \to 0} \frac{\log \mu(B(\xi,\rho))}{\log \rho} \]
where $B(\xi,\rho)$ is the spherical ball of radius $\rho$ and centre $\xi$. The number $\Dim(\lambda,Z)$ can be interpreted as the (almost sure) dimension of Patterson-Sullivan along chains (\emph{i.e.} boundaries of complex geodesics), with respect to a Gromov metric. See theorem C below in this introduction, and lemma \ref{lemma.def.dim} \emph{infra} for the precise definition of $\Dim(\lambda,Z)$. 

So the general lower bound 
\[ \sup \{ \delta_\Gamma - 1,\frac{\delta_\Gamma}{2} \} \leq \diminf(\mu,\xi) \]
(which holds by virtue of Balogh \emph{et al.} inequality) is made more precise; indeed the number $\dim(\lambda,Z)$ lies somewhere between $0$ and $\inf\{\delta_\Gamma,2\}$.

We actually improve slightly on the general result, indeed we prove
\begin{theoremB}[corollary \ref{cor.referee}]
Let $\Gamma$ be a non-elementary discrete group of $G$, Zariski-dense and geometrically finite; assume furthermore that $\Gamma$ is \emph{not} a lattice.  Then $\dim(\lambda,Z)<2$. In particular, we get the strict inequality
\[ \delta_\Gamma - 1 < \mathrm{dim}_E(\Lambda_\Gamma)  \]
where $\mathrm{dim}_E(\Lambda_\Gamma)$ is the Hausdorff dimension of the limit set with respect to the Euclidean metric on the boundary.
\end{theoremB}
Of course, this is better than the general inequality only if $\delta_\Gamma \geq 2$.

As an obvious corollary of theorem A we get the fact that if $\dim(\lambda,Z)=0$, then the Hausdorff dimension of the limit set $\Lambda_\Gamma$ must be \emph{equal} to $\delta_\Gamma$. In general, computing $\dim(\lambda,Z)$ seems to be a difficult problem. We will define a class of Schottky subgroups (which we call ``Schottky subgroups in good position'') and prove the following
\begin{theoremC}[corollary \ref{cor.schottky}]
Let $\Gamma$ be a Schottky subgroup in good position in $\mathbf{PU}(1,n)$. Then $\dim(\lambda,Z)=0$; in particular, the Hausdorff dimension of $\Lambda_\Gamma$, with respect to the spherical metric, is equal to $\delta_\Gamma$.
\end{theoremC}

As an intriguing consequence, we get the fact that if $\Gamma$ is a Schottky subgroup in good position, then $\delta_\Gamma \leq 2(n-1)$ (because by virtue of Balogh et al inequality, one must then have $\delta_\Gamma \leq \frac{\delta_\Gamma}{2}+n-1$). Wether one should expect general (Zariski dense) Schottky subgroups  to satisfy this inequality seems to be an interesting question.

In proving theorems A, B and C, we use the following version of Ledrappier-Young's well-known formula.
\begin{theoremD}[theorem \ref{th.additivity}]
Let $G$ be a metric locally compact second countable  group which acts in a Borel way on a standard Borel space $X$ with uniformly discrete stabilizers. Let $H$ be a closed normal subgroup of $G$ and assume that the metric group $G/H$ has the Besicovitch covering property. Let $\lambda$ be a Borel probability measure on $X$. Assume that the hypotheses stated at the beginning of section \ref{section.dimension} are satisfied, so that we may define the dimension of $\lambda$ along $G$ and $H$, $\Dim(\lambda,G)$ and $\Dim(\lambda,H)$ respectively, as well as the transverse dimension $\DimT(\lambda,G/H)$. Then the following holds:
\[ \Dim(\lambda,G)=\Dim(\lambda,H)+\DimT(\lambda,G/H) \text. \]
\end{theoremD}
We do not get into details as they are rather technical, see sections 2 to 4. We will apply this theorem to $X=\Gamma \backslash \mathbf{PU}(1,n)$, $G=N$ in some Iwasawa decomposition $\mathbf{PU}(1,n)=KAN$,  $H$ is the centre of $G$, and $\lambda$ is the Bowen-Margulis-Sullivan measure on $X$. In particular, $\Dim(\lambda,N)$ is equal to the growth exponent $\delta_\Gamma$.

Let us describe our strategy in proving theorem A. Let $\mu$ be the Patterson-Sullivan measure (of exponent $\delta_\Gamma$) associated with $\Gamma$ (recall that $\Gamma$ has finite BMS measure, so $\mu$ is essentially unique). We use the fact that the boundary $\partial \Hc^n$  is the one-point compactification of the Heisenberg group and we look at the way Patterson-Sullivan measure decomposes in Heisenberg space, with respect to the natural fibration of this space along the central direction. 

This is where the Ledrappier-Young formula enters the scene. According to this formula, the dimension of Patterson-Sullivan measure, with respect to the Gromov metric on the boundary (or, equivalently, with respect to the Heisenberg metric on the Heisenberg group) is equal to the sum $\Dim(\lambda,Z)+\DimT(\lambda,N/Z)$ where $\Dim(\lambda,Z)$ can be interpreted as the dimension of Patterson-Sullivan along the central direction, and $\DimT(\lambda,N/Z)$ is the dimension of Patterson-Sullivan ``transverse'' to this central direction.

We must then use this information to estimate the dimension of Patterson-Sullivan, now with respect the to spherical metric. Easy computations show that the ``transverse dimension'' is left unmodified, whereas the dimension along the central direction is divided by 2. The Ledrappier-Young formula holds no longer, but super-additivity of dimension does (proposition \ref{prop.additivity.1}). That's why in the end all we get is a lower bound.

The plan of this paper is as follows. In section 2 we sketch the basic theory of conditional measures along a group operation. In section 3 we define the dimension of conditional measures along a group operation, as well as transverse dimension (with respect to some normal subgroup) and we study the elementary relations between these numbers. In section 4 we state and prove a version of Ledrappier-Young formula in the setting of our theory. The most important result of sections 2 to 4 is theorem \ref{th.additivity}. Then in sections 5 and 6 we recall some facts of complex hyperbolic geometry and Patterson-Sullivan theory. In section 7 we apply our Ledrappier-Young theorem to study Hausdorff dimension of limit sets with respect to the spherical metric, as explained above. In section 8 we construct Schottky subgroups in good position of $\mathbf{PU}(1,n)$. In section 9 we discuss our results.

Most of our notations  are standard. If $\mu$ is some Borel measure on some metric space $X$, we let 
\[ \diminf(\mu,x) = \liminf_{\rho \to 0} \frac{\log \mu(B(x,\rho))}{\log \rho}, \quad \dimsup(\mu,x) = \limsup_{\rho \to 0} \frac{\log \mu(B(x,\rho))}{\log \rho} \text. \]
If $X,Y$ are measurable spaces (recall that a measurable space is a set endowed with a $\sigma$-algebra of subsets) and if $\phi : X \to Y$ is a measurable map, the push-forward of a measure $\lambda$ on $X$ is the measure $\phi_* \lambda$ defined by $\phi_* \lambda (B)=\lambda(\phi^{-1}(B))$ for all measurable subset $B$ of $Y$. If $\phi$ is a measurable isomorphism and $\mu$ a measure on $Y$, we denote the push-forward $(\phi^{-1})_* \mu$ by $\phi^* \mu$. 

If $G$ is a group, the right translation $h \mapsto hg^{-1}$ by some element $g \in G$ is denoted by $R_g$.

I would like to thank the referee for their useful remarks and for pointing out  (and correcting) a mistake in proposition 38. I also would like to thank my thesis advisor Jean-Fran\c{c}ois Quint for suggesting this problem and for his constant support.

\section{Disintegration of a measure along a group operation}
In this section we set up the ``language'' of conditional measures along a group operation. This is a very useful, albeit slightly technical, way of looking at things. The knowledgeable reader might want to skip it and jump straight into section 4, where the Ledrappier-Young formula we need is stated and proved. 

On the other hand the lay man (or woman) interested in knowing all the details could read the definitions below bearing in mind the following setting: $G$ is the real line acting in the usual way (i.e. as a linear flow) on the torus $X$. The stabilizers may be non-trivial but they are certainly discrete. We are going to disintegrate some finite Borel measure $\lambda$ carried by $X$ along the operation of $G$ (it is very important to note that we do not expect $\lambda$ to be $G$-invariant in any way). As is well known, if this operation has dense orbits (i.e. it is an irrational flow) the elementary theory of disintegration does not suit our needs. This is why we need the more sophisticated concept of conditional measures along a group operation.

In defining them, one is basically looking at some suspension of the dynamical systems ``$G$ acting on $X$'' in order to get rid of the troubles caused by ``wild orbits'' (to be more specific, the real concern is standard Borelness of the quotient space). This is why we introduce the so-called ``lacunary sections'' $\Sigma$. They allow us to lift our measure $\lambda$ to some quotient space $G \times \Sigma$ where $G$ acts \emph{smoothly} (\emph{i.e.} the quotient space is standard Borel). We lose finiteness of $\lambda$, while retaining $\sigma$-finiteness. We may then disintegrate the lifted measure appropriately, and check that the conditional measures thus obtained are indeed ``canonical'' in some way (i.e. they do not depend on the choice of $\Sigma$).

The reader looking for a more detailed account may want to look at our doctoral thesis \cite{thesis}.

\subsection{Definition}
We begin with some definitions and basic facts. First, we state a useful theorem of Kechris.

\begin{theorem}
Let $G$ be a locally compact second countable topological group, and $X$ a standard Borel space on which $G$ acts in a Borel way. Fix a compact symmetric neighbourhood $U$ of the identity in $G$. There exists a Borel subset $\Sigma \subset X$ such that
\begin{enumerate}
\item $\Sigma$ is a $U$-lacunary section: if $x' \in \Sigma$ and $g \in U$ are such that $gx' \in \Sigma$, then $gx'=x'$;
\item $\Sigma$ is $U^2$-complete: $U^2 \Sigma=X$.
\end{enumerate}
\end{theorem}
\begin{proof}
A close inspection of the proof of the main  theorem in \cite{Kechris2} yields this improved version.
\end{proof}
If there is no need to make the set $U$ explicit, we will use the phrase ``let $\Sigma$ be a complete lacunary section''. 

We will  make some assumptions on the operation of $G$ on $X$. If the stabilizer of any point $x \in X$ is a discrete subset of $G$, we say that the operation has \emph{discrete stabilizers}; in other words, for every $x \in X$ there is is a neighbourhood $V$ of identity in $G$ such that the only $g \in V$ which fixes $x$ is the identity. If the neighbourhood $V$ can be chosen independently of $x$, we say that the operation has \emph{uniformly discrete stabilizers}.

Now let $X$ be a standard Borel space, $\lambda$ a $\sigma$-finite Borel measure on $X$, and $G$ a locally compact second countable topological group acting  on $X$ in a Borel way with discrete stabilizers. We recall, following \cite{BenoistQuint}, how one can disintegrate the measure $\lambda$ along this group operation. 

Let $\Sigma$ be a complete lacunary section. The map 
\[ a_\Sigma : G \times \Sigma \to X\quad (g,x') \mapsto gx' \]
has countable fibers. Hence, we can define a Borel measure $a_\Sigma^* \lambda$ on $G \times \Sigma$ in the following way:
\[ a_\Sigma^* \lambda = \int d\lambda(x) \mathrm{Card}_{a_\Sigma^{-1}(x)} \]
where $\mathrm{Card}_{a_\Sigma^{-1}(x)}$ is the counting measure on the fiber $\{(g,x')\ ;\ gx'=x\}$. It is not difficult to check that $a_\Sigma^*\lambda$ is a $\sigma$-finite measure. We may thus choose some finite measure on $G \times \Sigma$, equivalent to $a_\Sigma^* \lambda$, and push it down $\Sigma$ (through the projection $G \times \Sigma \to \Sigma$) to get a finite Borel measure $\lambda_\Sigma$ (which we call a pseudo-image of $\lambda$). We may then disintegrate $a_\Sigma^* \lambda$ over $\lambda_\Sigma$:
\[ a_\Sigma^* \lambda = \int \mathrm{d} \lambda_\Sigma(x')\ \sigma_\Sigma(x') \otimes \delta_{x'}  \]
(where $\delta_{x'}$ is the Dirac measure supported on $\{x'\}$).

The conditional measures $\sigma_\Sigma(x')$ are (almost everywhere) uniquely defined up to a scalar constant. Replacing $\lambda_\Sigma$ with another pseudo-image amounts to replacing $\sigma_\Sigma(x')$ (for almost every $x'$) with some multiple $c(x') \sigma_\Sigma(x')$.

The following result is basic to the theory. We keep the standing notations and assumptions. Also, we denote by $\mathcal M_\sigma^1(G)$ (resp. $\Mrad(G)$) the space of projective classes of $\sigma$-finite (resp. Radon) non-zero measures on $G$. If $\mu$ is a non-zero measure, the projective class $[\mu]$ of $\mu$ is the equivalence class 
\[ [\mu] = \{ t \mu\ ;\ t>0\}\text. \]
The space $\mathcal M_{\mathrm{r}}(G)$ of all non-zero Radon measures on $G$ has a natural Borel structure (generated by narrow topology), and we endow the quotient space $\Mrad(G)$ with the quotient Borel structure. 
\begin{proposition}[\cite{BenoistQuint}, proposition 4.2] \label{prop.benoistquint}
There is a mapping $\sigma_{G,\lambda} : X \to \mathcal M^1_\sigma(G)$ with the following property: for any complete lacunary section $\Sigma$, there is a conegligible subset $X' \subset X$ such that for almost every $x' \in \Sigma$ and any $g \in G$, if $gx'$ belongs to $X'$ then
\[ R_g^*\sigma_{G,\lambda}(gx')=[\sigma_\Sigma(x')] \text. \]
The mapping $\sigma_{G,\lambda}$ is unique up to a negligible set. Furthermore, $\sigma_{G,\lambda}$ is essentially $G$-equivariant, \emph{i.e.} there is a conegligible set $X' \subset X$ such that if $x$ and $gx$ belong to $X'$, then
\[ \sigma_{G,\lambda}(gx)=(R_g)_* \sigma_{G,\lambda}(x) \text. \]
If $\lambda$ is finite, then $\sigma_{G,\lambda}$ maps $X$ into $\Mrad(G)$, and is Borel.
\end{proposition}

Note that for $x \in X$, $\sigma_{G,\lambda}(x)$ is not, strictly speaking, a measure on $G$. Nonetheless, in many situation there is no problem in dealing with $\sigma_{G,\lambda}(x)$ as if it were a genuine measure. Hopefully, this abuse of language will not cause too much harm to the reader.

The mapping $\sigma_{G,\lambda}$ will be called the \emph{disintegration of $\lambda$ along $G$}.

We will also need the following fact.
\begin{lemma} \label{lemma.finiteness}
Assume the operation of $G$ on $X$ has uniformly discrete stabilizers. Then for any lacunary section $\Sigma$ and any compact subset $K$ of $G$, the measure $a^*_\Sigma \lambda (K \times \Sigma)$ is finite.
\end{lemma}
\begin{proof}
Let $U$ be some relatively compact symmetric neighbourhood of the identity such that $\Sigma$ is $U^2$-lacunary. Since the operation has uniformly discrete stabilizers, we may shrink $U$ so that for any $x \in X$, the only $g$ in $U^2$ such that $gx=x$ is the identity. We deduce that for any $x \in X$, there is at most one pair $(g,x') \in U \times \Sigma$ such that $gx'=x$. Consequently, $a_\Sigma^* \lambda(U \times \Sigma) \leq 1$, and also $a_\Sigma^* \lambda (gU \times \Sigma) \leq 1$ for any $g \in G$. Now it becomes clear that $a_\Sigma^* \lambda (K \times \Sigma)$ must be finite for any compact subset $K$ of $G$.
\end{proof}

We state another lemma which we will use freely. We skip the easy proof.
\begin{lemma}
For $\lambda$-almost every $x \in X$, the identity element of $G$ belongs to the support of $\sigma_{G,\lambda}(x)$.
\end{lemma}

If $\Gamma$ is some discrete subgroup of finite Bowen-Margulis-Sullivan measure of $\mathbf{PU}(1,n)$ (say), let $X=\Gamma \backslash \mathbf{PU}(1,n)$, and let $G$ be the $N$ group in the Iwasawa decomposition $\mathbf{PU}(1,n)=KAN$. If we disintegrate the Bowen-Margulis-Sullivan measure $\lambda$ on the frame bundle $X$ along $N$, what we get is exactly the so-called ``horospheric measures'' (as seen on $N$). See \cite{Roblin}.

\subsection{Transitivity of disintegration}
We keep the notations and assumptions of the previous subsection. Let also $H$ be a closed subgroup of $G$. Recall that the operation of $G$ on $X$ has discrete stabilizers; obviously, the operation of $H$ on $X$ has discrete stabilizers as well. The disintegration of $\lambda$ along $G$ and $H$, respectively, gives two mappings
\[ \sigma_{G,\lambda} : X \to \mathcal M_\sigma^1(G),\quad \sigma_{H,\lambda} : X \to \mathcal M_\sigma^1(H) \text. \]
For any $x \in X$, we may also disintegrate $\sigma_{G,\lambda}(x)$ along the operation of $H$ on $G$ by left translation, thus obtaining a mapping
\[ \sigma_{H,\sigma_{G,\lambda}(x)} : G \to \mathcal M_\sigma^1(H) \]
 (\emph{stricto sensu} we are disintegrating some measure in the projective measure class $\sigma_{G,\lambda}(x)$ and the disintegration does not depend on the choice of this measure).

The meaning of the following proposition should be intuitive. The proof, however, is rather lengthy and technical. We skip it and refer the reader to our thesis, \cite{thesis}, section 2.1.3.
\begin{proposition} \label{prop.transitivity}
For $\lambda$-almost every $x \in X$ and $\sigma_{G,\lambda}(x)$-every $g \in G$, the following holds: 
\[ \sigma_{H,\sigma_{G,\lambda}(x)}(g)=\sigma_{H,\lambda}(gx) \text. \]
\end{proposition}

\section{Dimension of conditional measures} \label{section.dimension}
We do not pursue the most general theory of conditional measures along group operations any further. From now on, we will deal with ``hyperbolic transformations'' of the space $X$ that are equivariant with respect to $G$. This added structure allows us to define the dimension of conditional measures and prove some useful results, the most important of which is theorem 23 in section 4 below.

Let us state the standing hypotheses in this section: $G$ is a locally compact second countable topological group, acting in a Borel way on a standard Borel space $X$, with discrete stabilizers. The space $X$ carries a Borel probability measure $\lambda$. We disintegrate $\lambda$ along $G$, thus getting a mapping
\[ \sigma : X \to \mathcal M_{\mathrm{r}}^1(G) \text. \]
Now we make some further assumptions. Namely, we assume that we are given a Borel automorphism $\alpha_X : X \to X$  as well as a group automorphism $\alpha_G : G \to G$ such that the following conditions hold:
\begin{enumerate}
\item For every $x \in X$ and $g \in G$, $\alpha_X(gx)=\alpha_G(g) \alpha_X(x)$;
\item The automorphism $\alpha_X : X \to X$ leaves invariant the measure $\lambda$ (\emph{i.e.} the push-forward $(\alpha_X)_* \lambda$ is equal to $\lambda$);
\item The group $G$ is endowed with a compatible metric $d$ which is right $G$-invariant and such that $\alpha_G$ acts on $G$ as a contracting similitude, \emph{i.e.} there is some real constant $\alpha<1$, such that
\[ d(\alpha_G g,\alpha_G h)=\alpha d(g,h) \]
for any $g,h \in G$.
\item The measure $\lambda$ is $\alpha_X$-ergodic, \emph{i.e.} if $A$ is some Borel subset of $X$ such that $\alpha_X A=A$, then either $\lambda(A)=0$ or $\lambda(A)=1$.
\end{enumerate}

The reader should have in mind the following picture: $\alpha_X$ is some hyperbolic automorphism of $X$, and $G$ parametrizes a sub-foliation of the stable foliation associated with this hyperbolic automorphism. Starting from section 5, we will apply the ongoing theory to the following objects :
\begin{itemize}
\item $G$ is some connected closed normal subgroup of $N$ in the Iwasawa decomposition $\mathbf{PU}(1,n)=KAN$, acting on the right on the quotient space $X=\Gamma \backslash \mathbf{PU}(1,n)$ (where $\Gamma$ is some discrete subgroup of $\mathbf{PU}(1,n)$ with finite Bowen-Margulis-Sullivan measure); in fact $G$ will be either $N$ or its centre $Z$, but the theory would apply just as well to any connected closed subgroup containing $Z$ (recall that $Z$ is also the derived subgroup of $N$);
\item $\lambda$ is the Bowen-Margulis-Sullivan measure;
\item $\alpha_X$ is some non-trivial element of the real line acting on $\Gamma \backslash \mathbf{PU}(1,n)$ as the ``frame flow'';
\item $\alpha_G$ is the corresponding Heisenberg homothety
\item the metric $d$ carried by $G$ is the usual (restriction of) Heisenberg metric on the Heisenberg group $N$.
\end{itemize}

\subsection{Basic facts}
First note the following
\begin{lemma} \label{lemma.alphaequiv}
For $\lambda$-almost every $x$, 
\[ \sigma(\alpha_X x)=(\alpha_G)_* \sigma(x) \text. \]
\end{lemma}
\begin{proof}
This is a consequence of the ``uniqueness'' part in proposition \ref{prop.benoistquint}.
\end{proof}

We will need a technical definition and a few related facts.
\begin{definition}
A metric space $Y$ is called \emph{doubling} if there is a constant $N \geq 2$ such that any closed ball $B(y,r)$ can be covered by $N$ balls of radius $r/2$. The smallest of such numbers $N$ is the \emph{doubling constant} of $Y$.
\end{definition}
The relevance of this notion to our work is because of the following 
\begin{lemma}
Let $G$ be a locally compact topological group endowed with a right-invariant metric, and assume that $G$ admits a group automorphism $\alpha_G$ which is a similarity transformation with similarity ratio $<1$. Then $G$ is a doubling space.
\end{lemma}
\begin{proof}
To begin with, note that any closed ball in $G$ is compact. Indeed, since $G$ is locally compact we can find a radius $\rho>0$ small enough so that the closed ball $B(e,\rho)$ is a compact set. If $R>0$ is arbitrary, we have $\alpha_G^k B(e,R) \subset B(e,\rho)$ for any $k$ large enough; since $\alpha_G$ is continuous, this implies that $B(e,R)$ is itself a compact set. So we see that any bounded closed subset of $G$ must be compact.

We now prove that $G$ is a doubling space. Since $B(e,1)$ is a compact set and $B(e,\alpha/2)$ has non-void interior, we may find a finite sequence $g_1,\ldots,g_N \in G$ such that the right translates $B(e,\alpha/2)g_i$ cover $B(e,1)$ ($i=1,\ldots,n$). Now fix some real number $R>0$ and let $k$ be the integer such that $\alpha^{k+1} \leq R < \alpha^k$. We have $B(e,R) \subset B(e,\alpha^k)$, so $\alpha_G^{-k} B(e,R) \subset B(e,1)$, and we deduce that the right translates $B(e,\alpha^{k+1}/2) \alpha^k_G (g_i)$ ($i=1,\ldots,n$) cover $B(e,R)$. Since $\alpha^{k+1} \leq R$, we see that $B(e,R)$ can be covered by $N$ balls of radius $R/2$ (recall that the metric on $G$ is right invariant). Thus $G$ is a doubling space.
\end{proof}

A metric group satisfying the hypotheses of this lemma will be called \emph{a doubling group}.

\begin{proposition}[\cite{KaRaSu}, lemma 2.2]
Let $Y$ be a doubling metric space and $\mu$ a Borel probability measure on $Y$. There is a constant $D$, depending only on the doubling constant of $Y$, such that for $\mu$-almost every $y \in Y$, 
\[ \liminf_{\rho \to 0} \frac{\log \mu(B(y,\rho))}{\log \rho} \leq D \text. \]
\end{proposition}

The next lemma is both straightforward and basic to our work.
\begin{lemma} \label{lemma.def.dim}
Under the hypotheses stated at the beginning of section 3, for $\lambda$-almost every $x$, the conditional measure $\sigma(x)$ is \emph{exact-dimensional}, and its dimension does not depend on $x$. In other words, for $\lambda$-almost every $x \in X$ and $\sigma(x)$-almost every $g \in G$
\[ \lim_{\rho \to 0} \frac{\log \sigma(x)(B(g,\rho))}{\log \rho} \]
exists, is finite, and does not depend on $x$ nor on $g$. 
\end{lemma}
\begin{proof}
First we prove that the limit (maybe infinite) exists for almost every $x$ when $g$ is the identity of $G$. It is enough to show that 
\[ \lim_{n \to \infty} \frac{\log \sigma(x)(B(e,\alpha^n))}{n \log \alpha} \]
exists for $\lambda$-almost every $x$. Let
\[ I(x)= \log \frac{\sigma(x)(B(e,\alpha))}{\sigma(x)(B(e,1))} \text. \]
Note that $I(x) \leq 0$. By  lemma \ref{lemma.alphaequiv} and an obvious induction we get
\[ \sum_{k=0}^{n-1} I(\alpha_X^{-k} x)= \log \frac{\sigma(x)(B(e,\alpha^n))}{\sigma(x)(B(e,1))} \]
($n \geq 1$). Upon dividing both sides by $n \log \alpha$ and letting $n \to \infty$, we see that
\[ \lim_{n \to \infty} \frac{\log \sigma(x)(B(e,\alpha^n))}{n \log \alpha} = \lim_{n \to \infty} \frac{1}{n \log \alpha}\sum_{k=0}^{n-1} I(\alpha_X^{-k}x) \]
where the right side exists and is equal, for $\lambda$-almost every $x \in X$, to 
\[ \frac{1}{\log \alpha} \int_X I(x) \mathrm{d} \lambda(x) \in [0,+\infty] \]
by virtue of the ergodic theorem, since $\lambda$ is $\alpha_X^{-1}$-ergodic. Now the fact that this limit is almost surely finite is an obvious consequence of the previous proposition. Let $\delta\in [0,\infty[$ be this finite number, thus
\[ \lim_{\rho \to 0} \frac{\log \sigma(x)(B(e,\rho))}{\log \rho} = \delta \]
for every $x$ in some conull subset $X'$ of $X$. 

It is easy to check that, since $X'$ is conull, $gx$ must belong to $X'$ for $\lambda$-almost every $x \in X$ and $\sigma(x)$-almost every $g \in G$. Also, for $\lambda$-almost every $x$ and $\sigma(x)$-almost every $g \in G$, we have
\[ \sigma(gx) = (R_g)_* \sigma(x) \]
whence
\[ \delta = \lim_{\rho \to 0} \frac{\log \sigma(gx)(B(e,\rho))}{\log \rho} = \lim_{\rho \to 0} \frac{\log \sigma(x)(B(g,\rho))}{\log \rho} \]
because the metric carried by $G$ is right-invariant. The proposition is proven.
\end{proof}

\begin{definition}
The limit in the previous lemma is called \emph{dimension of $\lambda$ along $G$}, and is denoted by $\dim(\lambda,G)$.
\end{definition}

Note the following formula:
\begin{equation} \label{eq.dimension}
\Dim(\lambda,G) = \frac{1}{\log \alpha} \int_X \log \left( \frac{\sigma(x)(\alpha_H B)}{\sigma(x)(B)} \right) \mathrm{d}\lambda(x) 
\end{equation}
for any relatively compact neighbourhood $B$ of the identity.

We now state some useful facts in the setting of doubling metric spaces.
\begin{lemma}[\cite{KaRaSu}, proposition A.2]
Let $Y$ be a doubling metric space, and $\mu$ be a finite Borel measure on $Y$. Let $A$ be some Borel subset of $Y$ such that $\mu(A)>0$. For almost every $y \in A$, we have
\[ \limsup_{\rho \to 0} \frac{\mu(A \cap B(y,\rho))}{\mu(B(y,\rho))} = 1 \text. \]
\end{lemma}

Note that contrary to the well-known density theorem of Lebesgue (which holds in euclidian space and more generally in metric spaces satisfying the Besicovitch covering property) the left side is an upper limit, not a genuine limit.

\begin{proposition} \label{prop.additivity.1}
Let $Y$ be a complete doubling metric space, $Z$ a complete separable metric space, and $\pi : Y \to Z$ a Lipschitz mapping. Let $\mu$ be a Borel probability measure on $Y$, $\nu$ the push-forward of $\mu$ through $\pi$, and 
\[ \mu = \int_Z d\nu(z) \mu_z \]
be the disintegration of $\mu$ along $\pi$.

Assume that there exists a constant $\gamma \geq 0$ and a Borel mapping $\delta : Z \to [0,\infty[$ such that for $\mu$-almost every $y$ in $Y$ the following hold:
\[ \diminf(\mu_{\pi(y)},y) \geq \gamma,\quad \diminf(\nu,\pi(y)) \geq \delta(\pi(y)) \text. \]
Then for $\mu$-almost every $y \in Y$, we have
\[ \dimsup(\mu,y) \geq \gamma + \delta(\pi(y)) \text. \]
If, instead of assuming that $Y$ is doubling, we assume that this space is complete separable and satisfies the Besicovitch covering property, we get the stronger conclusion
\[ \diminf(\mu,y) \geq \gamma+\delta(\pi(y)) \]
for $\mu$-almost every $y$.
\end{proposition}
\begin{proof}
See \cite{LY} lemma 11.3.1 when $Y$ satisfies the Besicovitch covering property. If $Y$ is a doubling space just copy the proof and use the previous lemma instead of Lebesgue density theorem to obtain the weaker conclusion.
\end{proof}

\begin{remark} This is a key proposition for the main result, so it may be worth taking some time to comment on this inequality. Choose some continuous mapping $f: [0,1] \to [0,1]$ whose graph has Hausdorff dimension 2, see \emph{e.g.} \cite{Wingren}. Let $\nu$ be the Lebesgue measure on $Z=[0,1]$, and $\mu$ be the push-forward of $\nu$ through the mapping $x \mapsto (x,f(x))$, so that $\nu$ is itself the push-forward of $\mu$ through the projection onto the first factor. We may disintegrate $\mu$ above $\nu$, and obviously we get
\[ \mu=\int_0^1 d\nu(x) \delta_{(x,f(x))} \]
where $\delta_{(x,y)}$ is the Dirac mass at $(x,y)$. In particular, the previous proposition amounts to the obvious inequality
\[ 2 \geq 1\text. \]
This illustrates the fact that dimension  is only ``super-additive'', and we should not expect  equality to hold in general, unless some significant geometric assumption is made on the measures we are looking at.
\end{remark}

In the same way one proves the following
\begin{proposition} \label{prop.additivity.2}
Let $Y$ be a complete doubling metric space, $Z$ a standard Borel space, and $\pi : Y \to Z$ a Borel mapping. Let $\mu$ be a Borel probability measure on $Y$, $\nu$ the pushforward of $\mu$ through $\pi$, and 
\[ \mu = \int_Z d\nu(z) \mu_z \]
be the disintegration of $\mu$ along $\pi$. 

Assume that there exists a constant $\gamma \geq 0$ such that for $\mu$-almost every $y$ in $Y$ the following holds:
\[ \diminf(\mu_{\pi(y)},y) \geq \gamma \text. \]
Then for $\mu$-almost every $y \in Y$, we have
\[ \dimsup(\mu,y) \geq \gamma \text. \]
If, instead of assuming that $Y$ is doubling, we assume that this space is complete separable and satisfies the Besicovitch covering property, we get the stronger conclusion
\[ \diminf(\mu,y) \geq \gamma \]
for $\mu$-almost every $y$.
\end{proposition}
Note that here we do not assume that $Z$ is a Polish metric space, and even if this is so, $\pi$ might not be Lipschitz, so we are not just applying the previous proposition with $\delta=0$. Nonetheless the proof is similar and straightforward and we skip it.

\subsection{Monotonicity of dimension; transverse dimension}\label{subsection.monotonicity}
We keep the previous assumptions and we consider a closed subgroup $H$ of $G$ such that $\alpha_G(H)=H$, so that $\alpha_G$ induces an automorphism $\alpha_H$ of $H$. Obviously, the dimension of $\lambda$ along $H$ is well-defined; the aim of this subsection is to compare the two dimensions.

The following proposition is intuitively clear.
\begin{proposition} \label{prop.monotonicity}
We have
\[ \Dim(\lambda,G) \geq \Dim(\lambda,H) \text. \]
\end{proposition}
\begin{proof}
Let us disintegrate $\lambda$ along $G$ and $H$, thus obtaining mappings
\[ \sigma_G : X \to \mathcal \Mrad(G),\quad \sigma_H : X \to \mathcal \Mrad(H) \text. \]
Let $\pi : G \to G/H$ be the quotient mapping. Recall that $G/H$ is a locally compact second countable topological space, so that both $G$ and $G/H$ are standard Borel space and $\pi$ is a Borel mapping.

Fix some relatively compact neighbourhood $B$ of the identity element of $G$. For any $x \in X$, we let $\nu^B(x)$ be the Borel probability measure on $G$ obtained by conditioning $\sigma(x)$ on $B$, \emph{i.e.}
\[ \nu^B(x) = \frac{\sigma_G(x)|B}{\sigma_G(x)(B)} \text. \]

For $\lambda$-almost every $x$ in $X$, if we disintegrate the measure $\nu^B(x)$ along $\pi$, the conditional measure we get above $\pi(g)$ is, almost surely, proportional to the restriction of $\sigma_H(gx)$ on some neighbourhood of the identity element in $H$ (proposition \ref{prop.transitivity}). In particular, these conditional measures are almost surely exact dimensional, with dimension $\dim(\lambda,H)$.

By virtue of proposition \ref{prop.additivity.2}, we get that
\[ \dimsup(\nu^B(x),g) \geq \dim(\lambda,H) \]
for $\nu^B(x)$-almost every $g$ in $G$. Now $B$ is a neighbourhood of the identity element in $G$, so that $\nu^B(x)$ must be exact dimensional of dimension $\dim(\lambda,G)$, and we have thus proved that
\[ \dim(\lambda,G) \geq \dim(\lambda,H) \text. \]
\end{proof}

We will now improve on this result, and introduce, following Ledrappier and Young \cite{LY}, a \emph{transverse dimension}, under the supplementary assumption that $H$ is a normal subgroup. We can endow the topological quotient group $G/H$, which is locally compact and second countable, with a natural metric. More precisely, we let
\[ d(gH,g'H) = \inf\{ d(gh,g'h)\ ;\ h \in H \} \]
and we may check that this defines a right invariant metric, such that the quotient mapping
\[ \pi: G \to G/H \]
is $1$-Lipschitz, and that the group automorphism $\alpha_{G/H} : G/H \to G/H$ induced by $\alpha_G$ is a similitude with same ratio as $\alpha_G$ itself.

We may now define the transverse to $H$ dimension of $\lambda$ along $G$. If $B$ is some relatively compact neighbourhood of identity in $G$, we let, as in the proof of the previous proposition,
\[ \nu^B(x) = \frac{\sigma(x)|B}{\sigma(x)(B)},\quad \theta^B(x)=\pi_* \nu^B(x) \]
\emph{i.e.} $\nu^B(x)$ is the probability measure obtained by conditioning $\sigma(x)$ on $B$ and $\theta^B(x)$ is the push-forward of $\nu^B(x)$ through $\pi$.
\begin{proposition} \label{prop.transverse.dimension}
There is a finite number $\DimT(\lambda,G/H)$ such that the following holds. Let $B$ be a compact neighbourhood of identity in $G$. For $\lambda$-almost every $x \in X$, and $\theta^B(x)$-almost every $gH \in G/H$, 
\[ \liminf_{\rho \to 0} \frac{\log \theta^B(x) (B(gH,\rho))}{\log \rho} = \DimT(\lambda,G/H) \text. \]
In other words, $\theta^B(x)$ has lower pointwise dimension equal to $\DimT(\lambda,G/H)$ almost everywhere.
\end{proposition}
\begin{proof}
First, we prove that the lower dimension of $\theta^B(x)$ at the identity element of $G/H$ is constant $\lambda$-almost everywhere. In order to shorten notations we will write $B^T(\rho)$ for the set $\pi^{-1}(B(\pi(e),\rho))$ (where $\pi(e)=H$ is the identity element of $G/H$). Since the metric carried by $G$ is right-invariant, we have $B^T(\rho)=B(e,\rho)H$. Now let $\delta_B(x)$ be the lower dimension of $\theta^B(x)$ at the identity element, \emph{i.e.}
\[ \delta_B(x)=\liminf_{\rho \to 0}\frac{\log \theta^B(x)(B(\pi(e),\rho))}{\log \rho} = \liminf_{\rho \to 0} \frac{\log \sigma(x)(B \cap B^T(\rho))}{\log \rho} \text.\]

For $\lambda$-almost every $x$ we have (lemma \ref{lemma.alphaequiv})
\[ \delta_B(\alpha_X x)=\liminf_{\rho \to 0} \frac{\log \sigma(x)(\alpha_G^{-1} B \cap B^T(\alpha^{-1} \rho))}{\log \rho} \]
and the right side is obviously equal to $\delta_{\alpha_G^{-1} B}(x)$. Now assume for a moment that $B$ is an open ball $B(e,R)$; we then have $B \subset \alpha_G^{-1} B$, so that the relation $\delta_B(\alpha_X x)=\delta_{\alpha_G^{-1} B}x$ implies the relation $\delta_B(\alpha_X x) \leq \delta_B(x)$. By a straightforward application of Birkhov's ergodic theorem (bearing in mind the ergodicity of $\lambda$) we see that $\delta_B$ is almost surely constant, and the relation $\delta_B \circ \alpha_X = \delta_{\alpha_G^{-1} B}$ implies that the almost certain value of $\delta_B$ is equal to the almost certain value of $\delta_{\alpha_G^k B}$ for any integer $k$. Thus the almost certain value of $\delta_B$ does not depend on the radius $R$ if $B$ is the open ball $B(e,R)$; now if $B$ is any relatively compact neighbourhood of the identity, we can find radii $R'$ and $R''$ such that
\[ B(e,R') \subset B \subset B(e,R'') \]
and we get that $\delta_B$ is almost surely constant and that its almost certain value does not depend on $B$.

Denote by $\DimT(\lambda,G/H)$ the almost certain value of $\delta_B$ for any relatively compact neighbourhood $B$ of the identity. Now let $X'$ be a conull subset of $X$ such that $\delta_B(x)=\DimT(\lambda,G/H)$ for any relatively compact neighbourhood $B$ of the identity, and, furthermore, such that $\sigma(gx)=(R_g)_* \sigma(x)$ if $x$ and $gx$ belong to $X'$.

For any $x \in X'$ and any $g \in G$ such that $gx \in X'$, we have
\[ \diminf (\theta^B(x),\pi(g))=\liminf_{\rho \to 0} \frac{\log \sigma(x)(B \cap \pi^{-1}B(\pi(g),\rho))}{\log \rho} \text. \]
Also,  $\pi^{-1}(B(\pi(g),\rho))=B^T(\rho)g$ because $H$ is normal in $G$. The previous quantity is thus equal to
\[ \liminf_{\rho \to 0} \frac{\log \sigma(gx)(Bg^{-1} \cap B^T(\rho))}{\log \rho} \text. \]
Now if $g$ belongs to $B$, the set $Bg^{-1}$ is a relatively compact neighbourhood of the identity, so that we get
\[ \diminf(\theta^B(x),\pi(g))=\delta_{Bg^{-1}}(x) \]
and the right side is equal to $\DimT(\lambda,G/H)$. The proposition is proven.
\end{proof}

\begin{definition}
The quantity $\DimT(\lambda,G/H)$, whose existence was proven in the previous proposition, is called ``transverse to $H$ dimension of $\lambda$ along $G$''.
\end{definition}

\begin{proposition} \label{prop.subadd}
Under the previous hypotheses, the following holds:
\[ \Dim(\lambda,G) \geq \Dim(\lambda,H)+\DimT(\lambda,G/H) \text. \]
\end{proposition}
\begin{proof}
We just need to argue as in the proof of proposition \ref{prop.monotonicity}, applying proposition \ref{prop.additivity.1} in lieu of proposition \ref{prop.additivity.2}.
\end{proof}

\section{Ledrappier-Young formula}
This section is a translation of theorem C' in \cite{LY} into the language and setting of conditional measures along a group action.

\subsection{Two technical lemmata}
\subsubsection{Classical statements}
We first state two more-or-less classical lemmas, which we are going to generalize to the setting of conditional measures along a group operation. In proving the generalized results, we will use their ``classical'' counterparts.

I should mention that lemmas \ref{lemma.classical.1} and \ref{lemma.classical.2} are basic to the proof of theorem C' in (\emph{ibid.}). We are basically going to copy the arguments of Ledrappier and Young, in our language, simply replacing the ``classical'' lemmas with the ``generalized'' lemmas to follow .

Let $Y$ be some complete separable metric space satisfying Besicovitch covering property, and consider two Radon measures $\mu,\nu$. Assume, for simplicity, that $\mu$ is finite. For any $y \in Y$ and any radius $\rho > 0$, let 
\[ \phi_\rho(y) = \frac{\nu(B(y,\rho))}{\mu(B(y,\rho))} \in [0,+\infty] \]
and 
\[ \phi_*(y) = \inf_{\rho > 0} \phi_\rho(y) \text. \]

\begin{lemma} \label{lemma.classical.1}
The positive Borel function $- \log \phi_*$ is $\nu$-integrable, and 
\[ \int - \log \phi_* \mathrm{d} \nu \leq C(Y) \mu(Y) \]
where $C(Y)$ is the Besicovitch constant of $Y$.
\end{lemma}
\begin{proof}
Let $E_t$ be the set of all those $y$ such that $\phi_*(y) < e^{-t}$ ($t > 0$). If $y$ belongs to $E_t$, there is some radius $r(y)\in ]0,1[$ such that $\nu(B(y,r(y))) < e^{-t} \mu(B(y,r(y)))$. Let, by virtue of Besicovitch covering property, $A$ be a subset of $E_t$ such that the closed balls $B(y,r(y))$ ($y \in A$) cover $E_t$ with multiplicity less than $C(Y)$. We have
\[ \nu(E_t) \leq \sum_{y \in A} \nu(B(y,r(y))) \leq e^{-t} C(Y) \mu(Y) \text. \]
Now integrate
\[ \int_0^\infty \nu(E_t) \leq C(Y) \mu(Y) \int_0^\infty e^{-t} \mathrm{d}t \]
and the left side is equal to 
\[ \int - \log \phi_* \mathrm{d} \nu \]
by a classical application of Fubini theorem.
\end{proof}

\begin{lemma}[\cite{Mattila}, theorem 2.12] \label{lemma.classical.2}
If $\nu$ is absolutely continuous with respect to $\mu$, we have
\[ \lim_{\rho \to 0} \phi_\rho(y) = \frac{\mathrm{d} \nu}{\mathrm{d} \mu}(y) \]
for $\mu$-almost every $y$.
\end{lemma}

\subsubsection{Generalized statements}
We begin with the notations. Let $G$ be a locally compact second countable topological group acting in a Borel way on a standard Borel space $X$ with \emph{uniformly discrete stabilizers}. Let $H$ be a closed normal subgroup of $G$. We assume that the (metrizable) quotient group $G/H$ is endowed with a compatible metric $d$, right invariant and \emph{proper}, which means that any closed ball is a compact set. Note that we do not need to endow $G$ (nor $H$) with any metric. We denote by $B(gH,\rho)$ the open ball of radius $\rho$ in $G/H$. Let $\pi$ be the quotient map $G \to G/H$.

We make the assumption that the metric space $G/H$ satisfies the \emph{Besicovitch covering property} (see, \emph{e.g.}, \cite{Mattila}).

Now let $\lambda$ be some Borel probability measure on $X$. We disintegrate $\lambda$ along $G$ and $H$, thus obtaining maps
\[ \sigma_G : X \to \mathcal M_{\mathrm{r}}^1(G),\quad \sigma_H : X \to \mathcal M_{\mathrm{r}}^1(H) \text. \]

Fix some compact neighbourhood $A \subset B$ of the identity in $G$. For any $x \in X$ and $\rho > 0$ we let
\[ \phi_\rho(x) = \frac{\pi_* (\sigma(x)|A)(B(H,\rho))}{\pi_*(\sigma(x)|B)(B(H,\rho))} \text. \]
To shorten notations, we denote by $B^T_\rho(g)$ ($g \in G$)  the inverse image $\pi^{-1}(B(\pi(g),\rho))$, so that
\[ \phi_\rho(x) = \frac{\sigma(x)(B^T_\rho(e) \cap A)}{\sigma(x)(B^T_\rho(e) \cap B)} \text. \]
Let also
\[ \phi_*(x) = \inf_{\rho > 0} \phi_\rho(x) \text. \]

We know state:
\begin{lemma} \label{lemma.technical.1}
The integral
\[ \int - \log \phi_*(x) \mathrm{d} \lambda(x) \]
is finite.
\end{lemma}

\begin{lemma} \label{lemma.technical.2}
 For $\lambda$-almost every $x$, we have
\[ \lim_{\rho \to 0} \phi_\rho(x) = \frac{\sigma_H(x)(A)}{\sigma_H(x)(B)} \text. \]
\end{lemma}

\subsubsection{Proof of lemma \ref{lemma.technical.1}}
Fix some compact symmetric neighbourhood of identity $\Delta$ in $G$ such that $\Delta^4 \subset A$. Let $\Sigma$ be a $\Delta$-lacunary, $\Delta^2$-complete section to the operation of $G$ on $X$. Denote by $a$ the mapping $G \times \Sigma$, $a(g,x')=gx'$. Choose a pseudo-image $\lambda_\Sigma$ of $a^* \lambda$, and remember the notation $\sigma_\Sigma$ (\emph{cf supra}, subsection 2.1). We have
\[ \begin{split} \int -\log \phi_* \mathrm{d} \lambda & \leq \int d(a^* \lambda)(g,x') \mathbf{1}_{\Delta^2}(g) (- \log \phi_* (gx')) \\ & =\int d\lambda_\Sigma(x') \int_{\Delta^2} \mathrm{d}(\sigma_\Sigma(x'))(g) (-\log \phi_*(gx')) \text.
\end{split} \]
For $\lambda_\Sigma$-almost every $x'$ and $\sigma_\Sigma(x')$-almost every $g$, we have $\sigma(gx')=[(R_g)_* \sigma_\Sigma(x')]$. Fix such an $x'$. By Fubini theorem, using the fact that $\phi_*$ is less than $1$, we have
\[ \int_{\Delta^2} - \log \phi_*(gx') \mathrm{d}(\sigma_\Sigma(x'))(g) = \int_0^\infty \sigma_\Sigma(x') \left\{ g \in \Delta^2\ ;\ \phi_*(gx') < e^{-t} \right\} \mathrm{d}t \text. \]

Now let 
\[ \tilde \phi_\rho(x',g) = \frac{\sigma_\Sigma(x')(\Delta^2 \cap B_\rho^T(g))}{\sigma_\Sigma(x')(B \Delta^2 \cap B_\rho^T(g))} \]
and
\[ \tilde \phi_* (x',g) = \inf_{\rho >0} \tilde \phi_\rho(x',g)\text{.} \]
I claim that
\[ \tilde \phi_*(x',g) \leq \phi_*(gx') \]
for $\lambda_\Sigma$-almost every $x'$ and $\sigma_\Sigma(x')$-almost every $g \in \Delta^2$. Indeed, if $g \in \Delta^2$ is such that $\sigma(gx')=[(R_g)_* \sigma_\Sigma(x')]$, we have
\[ \begin{split} \phi_\rho(gx') &= \frac{\sigma(gx')(A \cap B_\rho^T(e))}{\sigma(gx')(B \cap B_\rho^T(e))} \\
& = \frac{\sigma_\Sigma(x')(Ag \cap B_\rho^T(g))}{\sigma_\Sigma(x')(Bg \cap B^T_\rho(g))} \\ & \geq \frac{\sigma_\Sigma(x')(\Delta^2 \cap B_\rho^T(g))}{\sigma_\Sigma(x')(B \Delta^2 \cap B_\rho^T(g))=\tilde \phi_\rho(g,x')}
\end{split}
\]
since $\Delta^2 \subset Ag$ and $Bg \subset B \Delta^2$ for any $g \in \Delta^2$. The  claim follows.

Now apply lemma \ref{lemma.classical.1} to get, for $\lambda_\Sigma$-almost every $x'$,
\[ \int_{\Delta^2} \mathrm{d} \sigma_\Sigma(x')(g) (-\log \tilde{\phi}_*(x',g)) \leq C(G/H) \sigma_\Sigma(x')(B \Delta^2) \]
where $C(G/H)$ is the Besicovitch constant of $G/H$. We need only integrate (with respect to $\mathrm{d} \lambda_\Sigma(x')$) to obtain
\[ \int - \log \phi_* \mathrm{d} \lambda \leq C(G/H) \int \sigma_\Sigma(x')(B \Delta^2) \mathrm{d} \lambda_\Sigma(x') \]
and the right side is finite by virtue of lemma \ref{lemma.finiteness}.

\subsubsection{Proof of lemma \ref{lemma.technical.2}}
\begin{enumerate}
\item Let $A' \subset B'$ be two relatively compact neighbourhood of the identity in $G$. Let $\Sigma$ be a lacunary section to the operation of $G$ on $X$. Recall the notation $\sigma_\Sigma$ (\emph{cf supra}). For $\lambda$-almost every $x'$ and $\sigma_\Sigma(x')$-almost every $g \in A'$, we have
\[ \lim_{\rho \to 0} \frac{\sigma_\Sigma(x')(A' \cap B_\rho^T(g))}{\sigma_\Sigma(x')(B' \cap B_\rho^T(g))} = \frac{\sigma_H(gx')(A' \cap H)}{\sigma_H(gx')(B' \cap H)} \text. \]
Indeed, this is a straightforward consequence of lemma \ref{lemma.classical.2} and proposition \ref{prop.transitivity}.
\item Now let, for $x$ in $X$, 
\[ \theta(x) = \limsup_{\rho \to 0} \left| \frac{\sigma(x)(A\cap B_\rho^T)}{\sigma(x)(B \cap B_\rho^T)} - \frac{\sigma_H(x)(A \cap H)}{\sigma_H(x)(B \cap H)} \right| \text. \]
We are going to show that for any $\varepsilon > 0$, the set of all $x$ such that $\theta(x) \geq \varepsilon$ is a null set (with respect to $\lambda$).
\item Fix some radius $r>0$, small enough in a way we will make precise soon. In order to shorten notations we denote by $U$ the open ball $B(e,r)$; $U$ is a relatively compact symmetric open neighbourhood of the identity. Let $\Sigma$ be a $U$-lacunary, $U^2$-complete section to the operation of $G$ on $X$. For almost every $x' \in \Sigma$ and $\sigma_\Sigma(x')$-almost every $g$, we have $\sigma(gx')=[(R_g)_* \sigma_\Sigma(x')]$; hence, for any $\rho>0$, 
\[ \frac{\sigma(gx')(A \cap B_\rho^T)}{\sigma(gx')(B \cap B_\rho^T)} = \frac{\sigma_\Sigma(x')(Ag \cap B_\rho^T(g))}{\sigma_\Sigma(x')(Bg \cap B_\rho^T(g))} \text. \]

We assume $r$ is small enough so that the sets
\[ A_{- r} = \bigcap_{g \in U^2} Ag,\quad B_{-r} = \bigcap_{g \in U^2} Bg \]
are neighbourhood of the identity. Let also
\[ A_{+r}=AU^2,\quad B_{+r}=BU^2 \]
and note that these sets are relatively compact.
\item For almost any $x' \in \Sigma$ and $\sigma_\Sigma(x')$-almost any $g \in U^2$, we have
\[ \frac{\sigma_{\Sigma}(x')(A_{-r} \cap B_\rho^T(g))}{\sigma_\Sigma(x')(B_{+r} \cap B_\rho^T(g))} 
\leq \frac{\sigma(gx')(A \cap B_\rho^T)}{\sigma(gx')(B \cap B_\rho^T)} \leq \frac{\sigma_{\Sigma}(x')(A_{+r} \cap B_\rho^T(g))}{\sigma_\Sigma(x')(B_{-r} \cap B_\rho^T(g))} \text. \]
By virtue of 1. \emph{supra}, we get
\begin{align*} & \frac{\sigma_H(gx')(A_{-r} \cap H)}{\sigma_H(gx')(B_{+r} \cap H)}  \leq \liminf_{\rho \to 0} \frac{\sigma(gx')(A \cap B_\rho^T)}{\sigma(gx')(B \cap B_\rho^T)} \\
\leq{} & \limsup_{\rho \to 0} \frac{\sigma(gx')(A \cap B_\rho^T)}{\sigma(gx')(B \cap B_\rho^T)} \leq \frac{\sigma_H(gx')(A_{+r} \cap H)}{\sigma_H(gx')(B_{-r} \cap H)} \text.
\end{align*}
\item Let
\[ \theta_r'(x) = \left| \frac{\sigma_H(x)(A \cap H)}{\sigma_H(x)(B \cap H)} - \frac{\sigma_H(x)(A_{-r} \cap H)}{\sigma_H(x)(B_{+r} \cap H)} \right| + \left| \frac{\sigma_H(x)(A \cap H)}{\sigma_H(x)(B \cap H)}-\frac{\sigma_H(x)(A_{+r} \cap H)}{\sigma_H(x)(B_{-r} \cap H)} \right| \text. \]
According to 4. we know that for almost every $x' \in \Sigma$ and $\sigma_\Sigma(x')$-almost every $g \in U^2$, we have $\theta(gx') \leq \theta_r'(gx')$. Since $\Sigma$ is $U^2$-complete, we obtain
\begin{align*} \lambda \{ x \in X \ ;\ \theta(x) \geq \varepsilon \} & \leq a^* \lambda \{ (g,x') \in U^2 \times \Sigma \ ;\ \theta(gx')  \geq \varepsilon \} \\
& \leq a^* \lambda \{ (g,x') \in U^2 \times \Sigma\ ;\ \theta_r'(gx') \geq \varepsilon \} \\
& = \int_{\theta_r' \geq \varepsilon} d\lambda(x) \kappa(x)
\end{align*}
where $\kappa(x)$ is the number of all $(g,x') \in U^2 \times \Sigma$ such that $gx'=x$. This number is bounded uniformly in $x$ by some constant $K$ independent of $r$ small enough; we skip the proof of this easy fact, which is a consequence of the facts that the stabilizers are uniformly discrete and that $G$ is a doubling group.

All in all, we have
\[ \lambda\{ x\ ;\ \theta(x) \geq \varepsilon \} \leq K\lambda\{ x\ ;\ \theta_r'(x) \geq \varepsilon \} \]
for any $r>0$ small enough. Now $\sigma_H(x)$ is (almost surely) a Radon measure, and $A,B$ are relatively compact open sets; whence 
\[ \lim_{r \to 0} \theta'_r(x)=0 \]
almost surely. The lemma is thus proven.
\end{enumerate}

\subsection{Additivity of dimension} \label{sect.LY}
We keep the notations and hypotheses set at the beginning of section \ref{section.dimension}. We also consider a closed normal subgroup $H$ of $G$ that is $\alpha_G$-invariant, \emph{i.e.} $\alpha_G H=H$. Recall that the quotient group $G/H$ is endowed with a natural metric, see section \ref{subsection.monotonicity}. We denote by $\pi$ the quotient map $G \to G/H$. Now we make two supplementary hypotheses:
\begin{enumerate}
\item The operation of $G$ on $X$ has uniformly discrete stabilizers.
\item The metric space $G/H$ satisfies the Besicovitch covering property.
\end{enumerate}

Under these hypotheses, we now state the main result of this section.
\begin{theorem} \label{th.additivity}
The following equality holds:
\[ \Dim(\lambda,G)=\Dim(\lambda,H)+\DimT(\lambda,G/H) \text. \]
\end{theorem}
\begin{proof}
Due to proposition \ref{prop.subadd}, we just need to establish
\[ \Dim(\lambda,G) - \Dim(\lambda,H) \leq \DimT(\lambda,G/H) \text. \]
Fix some open relatively compact neighbourhood $B$ of the identity in $G$. We must prove that the push-forward measure 
\[ \pi_* \left( \frac{\sigma_G(x)|B}{\sigma_G(x)(B)} \right) \]
has lower dimension $\geq \Dim(\lambda,G)-\Dim(\lambda,H)$ at the identity of $G/H$, for $\lambda$-almost every $x$ (proposition \ref{prop.transverse.dimension}).
Introduce the ``transverse ball''
\[ B_\rho^T(g)=\pi^{-1}(B(\pi(g),\rho)),\quad \rho>0,\quad g\in G \text. \] 
Note that $B_\rho^T(gg')=B_\rho^T(g)g'$ because the metric on $G$ is right-invariant. To shorten notations we denote the transverse ball at identity $B^T_\rho(e)$ by $B_\rho^T$. 

We only need to prove that for $\lambda$-almost every $x$,
\[ \liminf_{n \to \infty} \frac{1}{n \log \alpha} \log \left( \frac{\sigma(x)(B \cap B_{\alpha^n}^T)}{\sigma(x)(B)} \right) \geq \Dim(\lambda,G)-\Dim(\lambda,H)\text. \]

The key of the argument is the relation
\[ \frac{\sigma(x)(B \cap B_{\alpha^n}^T)}{\sigma(x)(B)} 
= \frac{\sigma(x)(B \cap B_{\alpha^n}^T)}{\sigma(x)(\alpha_G B \cap B_{\alpha^n}^T)} \times \frac{\sigma(x)(\alpha_G B)}{\sigma(x)(B)} \times \frac{\sigma(x)(\alpha_G B \cap B_{\alpha^n}^T)}{\sigma(x)(\alpha_G B)}\text{.} \]

This relation implies, by an obvious induction (recall that $\sigma(\alpha x)=\alpha_* \sigma(x)$ almost everywhere and $\alpha_{G/H} B_\rho^T=B_{\alpha \rho}^T$) that for any $p, n \geq 1$, 
\[ \begin{split} \frac{\sigma(x)(B \cap B_{\alpha^n}^T)}{\sigma(x)(B)} & = \left( \prod_{i=0}^{p-1}  \frac{\sigma(\alpha_X^{-i}x)(B \cap B_{\alpha^{n-i}}^T)}{\sigma(\alpha_X^{-i}x)(\alpha B \cap B_{\alpha^{n-i}}^T)} \times \frac{\sigma(\alpha_X^{-i}x)(\alpha B)}{\sigma(\alpha_X^{-i}x)(B)} \right)\times\\& \frac{\sigma(\alpha_X^{-p}x)(B \cap B_{\alpha^{n-p}}^T)}{\sigma(\alpha_X^{-p}x)(B)} \text.
\end{split}
\]

If $p \leq n$, the last factor on the right-hand side is obviously less than $1$. Now let 
\[ \phi_\rho(x) = \log \frac{\sigma(x)(\alpha_G B \cap B_\rho^T)}{\sigma(x)(B \cap B_\rho^T)}, \tau(x) = \log \frac{\sigma(x)(\alpha_G B)}{\sigma(x)(B)}, \phi_*(x) = \inf_{\rho > 0} \phi_\rho(x) \]
so that for any $p \leq n$ we have
\begin{equation} \label{eq.dem.1} \frac{1}{n \log \alpha} \log \frac{\sigma(x)(B \cap B_{\alpha^n}^T)}{\sigma(x)(B)} \geq \frac{1}{n \log \alpha} \sum_{i=0}^{p-1} \tau(\alpha_X^{-i}x) -  \frac{1}{n\log \alpha} \sum_{i=0}^{p-1} \phi_{\alpha^{n-i}}(\alpha_X^{-i}x)\text.
\end{equation}

Fix some small $\varepsilon > 0$ and let $p_n$ be the integral part of $(1-\varepsilon)n$, so that the numbers $\alpha^{n-i}$ converge to $0$ as $n$ tends to infinity, and uniformly so with respect to $i \leq p_n$.

Let us take care of the first member of the right-hand side. We have
\begin{equation} \liminf_{n \to \infty} \frac{1}{n \log \alpha} \sum_{i=0}^{p_n-1} \tau(\alpha_X^{-i} x) = (1-\varepsilon) \int_X \frac{1}{\log \alpha} \log \left( \frac{\sigma(x)(\alpha_G B)}{\sigma(x)(B)} \right) = (1- \varepsilon) \Dim(\lambda,G) 
\end{equation}
by virtue of Birkhov's ergodic theorem and formula \eqref{eq.dimension}.

We now handle the remaining member of the right-hande side. For $\lambda$-almost every $x$ there is a real number $\rho(x)>0$ such that for any $\rho < \rho(x)$ there holds
\[ \phi_\rho(x) \geq \log\left( \frac{\sigma_H(x)(\alpha_G B)}{\sigma_H(x)(B)} \right) - \varepsilon \]
(lemma \ref{lemma.technical.2}). Choose $\rho_0>0$ small enough that, letting 
\[ M = \{x\in X\ ;\ \rho(x)>\rho_0\} \]
we have
\[ \int_{X \setminus M} \phi_* \mathrm{d} \lambda \geq -\varepsilon \]
(lemma \ref{lemma.technical.1}). We split the sum on the left-hand side of \eqref{eq.dem.1} over indices $i$ such that $\alpha_X^{-i} x \in M$ and indices $i$ such that $\alpha_X^{-i} x \notin M$. First,
\[\begin{split} \liminf_{n \to \infty}& \frac{-1}{n \log \alpha} \sum_{0 \leq i \leq p_n}  \mathbf{1}_M(\alpha_X^{-i}x) \phi_{\alpha^{n-i}}(\alpha_X^{-i}x) \geq  \\ & \liminf_{n \to \infty} \frac{-1}{n \log \alpha} \sum_{0 \leq i \leq p_n} \mathbf{1}_M(\alpha_X^{-i}x) \left( \log \left( \frac{\sigma_H(\alpha_X^{-i}x)(\alpha B)}{\sigma_H(\alpha_X^{-i}x)(B)} \right) - \varepsilon \right) \\
& = -(1-\varepsilon)\frac{1}{\log \alpha} \int_M \left( \log \left(\frac{\sigma_H(x)(\alpha_G B)}{\sigma_H(x)(B)}\right) - \varepsilon\right) d\lambda(x)\\ & \geq - \Dim(\lambda,H)  + \frac{\varepsilon}{\log \alpha} \text.
\end{split}
\]
Second, 
\[ \begin{split} 
\liminf_{n \to \infty}& \frac{1}{n} \sum_{0 \leq i \leq p_n} \mathbf{1}_{\complement M}(\alpha_X^{-i} x) \phi_{\alpha^{n-i}}(\alpha_X^{-i}x)\\ & \geq \liminf_{n \to \infty} \frac{1}{n} \sum_{0 \leq i \leq p_n} \mathbf{1}_{\complement M}(\alpha_X^{-i}x) \phi_*(\alpha_X^{-i}x) \\ & = (1-\varepsilon) \int_{\complement M} \phi_* d\lambda \geq -\varepsilon(1-\varepsilon) \text. 
\end{split}
\]

All in all, we get, for all $\varepsilon>0$, the inequality
\[ \liminf_{n \to \infty} \frac{1}{n \log \alpha} \log \left( \frac{\sigma(x)(B \cap B_{e^{-n}}^T)}{\sigma(x)(B)} \right) \geq (1-\varepsilon) \Dim(\lambda,G) - \Dim(\lambda,H) + \frac{\varepsilon + \varepsilon(1-\varepsilon)}{\log \alpha} \]

The theorem is thus proved.
\end{proof}
\section{Complex hyperbolic spaces}
\subsection{Basic facts}
Fix an integer $n \geq 1$. We denote by $G$ the group $\mathbf{PU}(1,n)$, \emph{i.e.} the projective unitary group associated with an hermitian form of signature $(1,n)$. We fix once and for all an Iwasawa decomposition $G=KAN$. Here, $K$ is a maximal compact subgroup of $G$, $A$ is a Cartan subgroup, and $N$ a maximal unipotent subgroup. We identify the quotient $G/K$ with complex hyperbolic space of (complex) dimension $n$, $\Hc^n$. The hyperbolic metric of $\Hc^n$ is denoted by $d$. The group of isometric transformations of $\Hc^n$ is equal to $G$. The complex hyperbolic space has sectional curvature lying between $-4$ and $-1$.

We denote by $\partial \Hc^n$ the boundary at infinity of $\Hc^n$. Recall that $\Hc^n$ is homeomorphic to the $2n$-dimensional ball $\mathbf{B}^{2n}$ and thus $\partial \Hc^n$ is homeomorphic to the $(2n-1)$-dimensional sphere $\mathbf{S}^{2n-1}$.

A holomorphic totally geodesic submanifold of (complex) dimension one is called a \emph{complex geodesic}. These objects play a major role in the questions we will be interested in. The boundary of a complex geodesic is called a \emph{chain}. Recall the following easy, although important, fact:
\begin{proposition}[\cite{Goldman}, theorem 3.1.11]
\begin{enumerate}
\item Any pair of distinct points in $\Hc^n \cup \partial \Hc^n$ lies on a unique complex geodesic.
\item Given a nonzero tangent vector $v$ in the tangent space $T_x \Hc^n$ there is a unique complex geodesic containing $x$ and tangent to $v$.
\item Any pair of distinct points of the boundary $\partial \Hc^n$ lies on a unique chain.
\end{enumerate}
\end{proposition}

Recall also that the boundary $\partial \Hc^n$ carries a natural CR-structure. We will not explicitly make use of this CR-structure but we note the following fact, in relation with the previous theorem.
\begin{proposition}[\cite{Goldman}, theorem 4.3.5]
Let $\xi \mapsto P_\xi$ be the natural CR-structure on the boundary $\partial \Hc^n$. If $\xi$ belongs to $\partial \Hc^n$ and $v \in T_\xi \partial \Hc^n$ is a non-zero tangent vector at $\xi$, there is a chain passing through $x$ and tangent to $v$ if and only if $v$ does not belong to $P_\xi$.
\end{proposition}
\subsection{Heisenberg group and the boundary at infinity}
Recall that $G=KAN$ is a fixed Iwasawa decomposition. The unipotent group $N$ is isomorphic to Heisenberg group which we now introduce.

Let $V$ be a real vector space of dimension $2(n-1)$, and let $\omega$ be a non-degenerate skew-symmetric bilinear form on $V$. It is well known that two such forms are always conjugate, so we may agree (for example) that $V=\R^{2(n-1)}$ and 
\[ \omega(x,y) = \sum_{i=1}^{n-1} x_{2i-1} y_{2i}-x_{2i} y_{2i-1} \text. \]
The product set $\mathcal H_n = V \times \R$ endowed with the group law
\[ (v,s) \cdot (w,t) = (v+w, s+t+\omega(v,w)) \]
and the product topology is a locally compact second countable topological group called the \emph{Heisenberg group}. The centre $Z$ of $\mathcal H_n$ is $\{0\} \times \R$, and is also equal to the derived subgroup.

We now define the \emph{Heisenberg metric} on $\mathcal H_n$ in the following way. First, let $\| \cdot \|_H$ be the ``Heisenberg pseudo-norm''
\[ \| (v,t) \|_H = (\|v\|^4+t^2)^{1/4} \]
(where $\|v\|$ is the euclidean norm on $\R^{2(n-1)}$). Then let
\[ d_H((v,t),(w,s)) = \| (v,t) \cdot (w,s)^{-1} \|_H \]
($v,w \in \mathcal H_n$). Obviously, $d_H$ is a right-invariant metric. If we denote by $h_\lambda$ ($\lambda \in \C^*$) the transformation
\[ h_\lambda(v,t) = (\lambda v,|\lambda|^2 t) \]
we define a similitude transformation of $\mathcal H_n$ with similitude ratio $|\lambda|$, and $h_\lambda$ is a group automorphism as we readily check. Such a transformation is called a \emph{Heisenberg similitude}. 

Let us return to the Iwasawa decomposition $G=KAN$. The group $AN$ fixes a unique point $\xi_+ \in \partial \Hc^n$, and $A$ itself fixes another point $\xi_- \in \partial \Hc^n$. We look at the operation of $N$ on $\partial \Hc^n$. This operation is transitive on the complement of $\{ \xi_+\}$. The mapping $n \mapsto n \xi_-$ ($n \in N$) is a homeomorphism $\phi$ from $N$ onto $\partial \Hc^n \setminus \{\xi_+ \}$, and this homeomorphism is the restriction of a homeomorphism from the Alexandrov compactification $N \cup \{\infty\}$ onto $\partial \Hc^n$.

The centre $Z$ of $N$ is mapped by $\phi$ onto the chain passing through $\xi_-$ and $\xi_+$. More generally, $\phi$  maps the translates of $Z$ in $N$ (\emph{i.e.} the fibers of the quotient mapping $N \to N/Z$) onto chains passing through $\xi_+$. 

In other words, $\phi$ defines a quotient bijection from $N/Z$ onto the sets of all chains passing through $\xi_+$. 

\subsection{Metrics on the boundary at infinity}
Recall that the boundary $\partial \Hc^n$  is diffeomorphic to the $(2n-1)$-sphère $\mathbf{S}^{2n-1}$. We can endow $\partial \Hc^n$ with the usual ``spherical metric''. Of course, the metric itself depends on the choice of the diffeomorphism, but the bilipschitz equivalence class is uniquely defined, because the sphere is compact. In other words, if $d$ and $d'$ are two spherical metrics on $\partial \Hc^n$, there exist a uniform constant $C>0$ such that
\[ \frac{1}{C} d(\xi,\eta) \leq d'(\xi,\eta) \leq C d(\xi,\eta) \]
for any $\xi,\eta \in \partial \Hc^n$. In particular, any notion invariant under scale changes (for example local lower or upper dimension of some measure at some point) does not depend on the choice of a spherical metric.

Now there is another bilipschitz equivalence class of metrics on $\partial \Hc^n$, less elementary albeit more ``natural'' in some sense. Recall that $\Hc^n$ has pinched negative curvature, so that we may consider Busemann functions
\[ b_\xi(x,y) = \lim_{t \to \infty} d(x,\xi_t)-d(y,\xi_t), \quad x,y \in \Hc^n,\quad \xi \in \partial \Hc^n \]
where $\xi_t$ parametrizes some unit-speed geodesic such that $\lim_{t \to \infty} \xi_t = \xi$. For any $x \in \Hc^n$, the \emph{Gromov metric} from $x$ between $\xi$ and $\eta$ in $\partial \Hc^n$ is then
\[ d_x(\xi,\eta) = e^{-\frac{1}{2} \left( b_\xi(x,p)+b_\eta(x,p) \right)} \]
where $p$ is some point on the (real) geodesic from $\xi$ to $\eta$. This metric is compatible with the topology of $\partial \Hc^n$. The point is that the family $(d_x)_{x \in \Hc^n}$ satisfies the following properties: for any $x,y \in \Hc^n$, $\xi,\eta \in \partial \Hc^n$ and $g \in G$,
\begin{enumerate}
\item Conformality: $d_y(\xi,\eta)=e^{\frac{1}{2}(b_\xi(x,y)+b_\eta(x,y))} d_x(\xi,\eta)$
\item Equivariance: $d_{gx}(g\xi,g\eta)=d_x(\xi,\eta)$
\end{enumerate}

Conformality property above makes it clear that all metrics $(d_x)_{x \in \Hc^n}$ are pairwise bilipschitz equivalent.

The following lemma results from easy computations.
\begin{lemma} \label{lemma.gromov.spherical}
Gromov metrics are not bilipschitz equivalent to spherical metrics. Nonetheless, if $d_G, d_E$ are a Gromov and a spherical, respectively, metric on the boundary, there is a constant $C>0$ such that
\[ \frac{1}{C}  d_G(\xi,\eta)^2 \leq d_E(\xi,\eta) \leq C d_G(\xi,\eta) \]
for any $\xi,\eta \in \partial \Hc^n$.
\end{lemma}
In fact, the boundary $\partial \Hc^n$ has dimension $2n-1$ with respect to spherical metrics, and $2n$ with respect to Gromov metrics. More generally, we recall the following
\begin{theorem}\label{th.balogh}
Let $S$ be some subset of $\partial \Hc^n$ and let $\alpha,\beta$ be the Hausdorff dimensions of $S$ with respect to spherical metrics, and Gromov metrics, respectively. Then
\[ \max\{\alpha,2\alpha-2n\} \leq \beta \leq \min\{2 \alpha,\alpha + 1 \} \]
and these inequalities are sharp.
\end{theorem}
\begin{proof}
See \cite{BaloghTyson} theorem 2.4, and lemma \ref{lemma.loclip} \emph{infra}.
\end{proof}

Before ending this subsection, let us state for future reference the following
\begin{lemma} \label{lemma12}
Let $d_H$ and $d_E$ be the Heisenberg metric and Euclidean metric respectively on Heisenberg space $\mathcal H_n$. 
\begin{enumerate}
\item For any compact subset $K \subset \mathcal H_n$, there is a constant $C>1$ such that
\[ \frac{1}{C} d_H(h,h')^2 \leq d_E(h,h') \leq C d_H(h,h') \]
for any $h,h' \in K$.
\item If $h,h' \in \mathcal H_n$ are such that $h-h'$ belongs to the centre, then 
\[ d_E(h,h')=d_H(h,h')^2 \text. \]
\item The quotient, on $\mathcal H_n/Z$, of the Heisenberg metric, is equal to the quotient of the euclidean metric.
\end{enumerate}
\end{lemma}
\begin{proof}
Straightforward computations.
\end{proof}

\subsection{The unit tangent bundle}
Let $T^1 \Hc^n$ be the unit tangent bundle of $\Hc^n$, that is, the set of all pairs $(x,u)$, where $u \in \Hc^n$ and $u \in T_x \Hc^n$ is a unitary tangent vector at $x$. The operation of $G$ on $\Hc^n$  extends to an operation on $T^1 \Hc^n$. One may identify $T^1 \Hc^n$ with the quotient space $G/M$. Here, $M$ is the centralizer, in $K$, of $A$ (remember that we fixed once and for all an Iwasawa decomposition $G=KAN$).

 If $u$ belongs to $T^1 \Hc^n$, we denote by $u^+$ and $u^-$ the forward and backward, respectively, endpoints in $\partial \Hc^n$ of the geodesic defined by $u$. Likewise, if $gM$ belongs to $G/M$, we denote by $g^+$ and $g^-$ the points $u^-$ and $u^-$, respectively, where $u$ is the element of $T^1 \Hc^n$ corresponding to $gM$.

Since $A$ is included in the normalizer of $M$, the operation of $A$ on $G$ by translation on the right induces an operation of $A$ on $G/M$, \emph{i.e.} $(gM,a) \mapsto gaM$. There is a (unique) isomorphism $\R \to A$, $t \mapsto a_t$, such that this operation of $A$ on $G/M$ is identified with the geodesic flow on $T^1 \Hc^n$. 

On the contrary, $N$ does not normalize $M$, so the operation of $N$ on $G$ by translation on the right does not give rise to an operation on $G/M$. Note, though, that $M$ does normalize $N$, so that $N$-orbits are well-defined in $G/M$. Actually, for any $v \in G/M$, the $N$-orbit $vN$ id the unstable manifold passing through $v$.

We need to introduce the so-called Hopf coordinates on $T^1 \Hc^n$, \emph{i.e} the mapping 
\[ T^1 \Hc^n \to \partial^2 \Hc^n \times \R,\quad u \mapsto (u^-, u^+, b_{u^-}(u,o)) \]
where we denote, as is customary, by $\partial^2 \Hc^n$ the set of all pairs $(\xi,\eta) \in \partial \Hc^n \times \partial \Hc^n$ with $\xi \neq \eta$, and $o$ is some fixed ``base point'' in $\Hc^n$. This is a diffeomorphism. To shorten notations, we will write $u=(\xi,\eta,s)$ if $\xi=u^-$, $\eta=u^+$ and $s=b_\xi(u,o)$.

We will need the following facts.
\begin{lemma}[\cite{PaulinHersonsky}, appendix] \label{lemma.loclip}
For any $g \in G$, let $\phi_g : N \to \partial \Hc^n \setminus \{ g^-\}$ be the mapping $n \mapsto (gn)^+$.
\begin{enumerate}
\item If $N$ is endowed with Heisenberg metric and $\partial \Hc^n$ with a Gromov metric, $\phi_g$ is locally bilipschitz.
\item If $N$ is endowed with euclidean metric and $\partial \Hc^n$ with a spherical metric, $\phi_g$ is locally bilipschitz.
\end{enumerate}
\end{lemma}

\section{Patterson-Sullivan theory}
\subsection{Limit set and growth exponent}
A good reference for this section is \cite{Roblin}. We keep the notations and conventions of the previous section. Let $\Gamma$ be a discrete subgroup of $G$. If $x$ is some point of $\Hc^n$, the set of accumulation points of the orbit $\Gamma \cdot x$ on $\Hc^n \cup \partial \Hc^n$ is a subset $\Lambda_\Gamma$ of the boundary, namely $\Lambda_\Gamma = \overline{\Gamma \cdot x} \cap \partial  \Hc^n$. This set does not depend on $x$. It is called the limit set of $\Gamma$. If $\Lambda_\Gamma$ a finite set, $\Gamma$ is called elementary, otherwise $\Gamma$ is called non-elementary.

The growth exponent of $\Gamma$,
\[ \delta_\Gamma = \limsup_{R \to \infty} \frac{1}{R} \mathrm{Card} \{\gamma \in \Gamma\ ;\ d(x,\gamma x) \leq R \} \]
does not depend on $x$. It is a finite number, $0 < \delta_\Gamma \leq 2n-1$. 

The study of $\Lambda_\Gamma$ and $\delta_\Gamma$ goes back a long way. We state the following important result, though we will not make use of it.
\begin{theorem}[\cite{BishopJones}, \cite{Stratmann}, \cite{PaulinBJ}]
Assume $\Gamma$ is non-elementary. Then the subset $\Lambda_\Gamma^c \subset \Gamma$ of conical limit points of $\Gamma$ has Hausdorff dimension $\delta_\Gamma$ with respect to Gromov metrics.
\end{theorem}
Recall that a point $\xi \in \partial \Hc^n$ is a conical limit point if there is an infinite sequence $(\gamma_n)$ of (pairwise distincts) elements of $\Gamma$ such that the distance from $\gamma_n x$ to the geodesic $]x,\xi[$ is bounded (uniformly in $n$), for some $x$ (and, thus, for any $x$).

This theorem raises the following 
\begin{question}
Let $\Gamma$ be a discrete non-elementary subgroup of $G$. What is the Hausdorff dimension of $\Lambda_\Gamma^c$ with respect to the spherical metric on the boundary?
\end{question}

\subsection{Conformal densities}
\begin{definition}
Let $\Gamma$ be a non-elementary discrete subgroup of $G$. Let $\beta$ be some real number $\geq 0$. A $\Gamma$-conformal density of exponent $\beta$ is a family $(\mu_x)_{x \in \Hc^n}$ of finite measures on $\partial \Hc^n$ which satisfies 
\begin{enumerate}
\item $\Gamma$-equivariance: 
\[ \gamma_* \mu_x = \mu_{\gamma x} \]
for any $x \in \Hc^n$ and any $\gamma \in \Gamma$.
\item Conformality: for any $x,y \in \Hc^n$, $\mu_x$ and $\mu_y$ are equivalent measures and the Radon-Nikodym derivative is given by
\[ \frac{\mathrm{d} \mu_y}{\mathrm{d} \mu_x}(\xi) = e^{-\beta b_\xi(y,x)} \]
almost everywhere.
\end{enumerate}
\end{definition}

The following well-known theorem is basic.
\begin{theorem}[\cite{Roblin}]
Let $\Gamma$ be a non-elementary discrete subgroup of $G$, with growth exponent $\delta_\Gamma$. There exist a $\Gamma$-conformal density of exponent $\delta_\Gamma$.
\end{theorem}

\subsection{Bowen-Margulis-Sullivan measure}
Now, let $\Gamma$ be a discrete non-elementary subgroup of $G$. Let $\mu$ be a $\Gamma$-conformal density of exponent $\delta_\Gamma$. Fix some arbitrary point $x \in \Hc^n$. We define the Bowen-Margulis-Sullivan measure $\BMS$ on $T^1 \Hc^n$:
\begin{equation} \label{eq.BMS}
\mathrm{d} \BMS(u) = e^{\delta_\Gamma(b_\xi(x,u)+b_\eta(x,u))} \mathrm{d} \mu_x(\xi) \mathrm{d} \mu_x(\eta) \mathrm{d}s
\end{equation}
where  $u=(\xi,\eta,s)$.
This Radon measure does not depend on the choice of $x$. It is invariant under the geodesic flow as well as under $\Gamma$. Consequently, the measure on $\Gamma \backslash T^1 \Hc^n$ defined by passing to the quotient is a Radon measure, invariant under the geodesic flow. Remember that $T^1 \Hc^n$ is identified with $G/M$, so that we get a Radon measure on $\Gamma \backslash G/M$, invariant under the operation of $A$ on the right. 
\begin{definition}
We say that a discrete non-elementary subgroup $\Gamma$ of $G$ has finite BMS measure if the associated Bowen-Margulis-Measure on $\Gamma \backslash T^1 \Hc^n$ (or $\Gamma \backslash G/M$) is finite.
\end{definition}

\begin{theorem}[\cite{Roblin}]
Let $\Gamma$ be a discrete non-elementary subgroup of $G$. If $\Gamma$ has finite BMS measure, the $\Gamma$-conformal density of exponent $\delta_\Gamma$ is unique, atomless, its support is the limit set $\Lambda_\Gamma$, and the conical limit set $\Lambda_\Gamma^c$ has full measure. Furthermore, the BMS measure is (strongly) mixing with respect to the geodesic flow.
\end{theorem}

The Bowen-Margulis-Measure on $\Gamma \backslash G/M$ does not exactly suit our needs, because, as we said, $N$ does not act on the right on this space. Hence we are lead to consider the unique $M$-invariant lifting of this measure to $\Gamma \backslash G$. We still call this measure on $\Gamma \backslash G$ the Bowen-Margulis-Measure. Note that the right action of $A$ on $\Gamma \backslash G/M$ extends to a right action on $\Gamma \backslash G$. The space $\Gamma \backslash G$ is sometimes called \emph{the frame bundle of $\Gamma \backslash \Hc^n$}, and the operation of $\R$ on $\Gamma \backslash G$, $(t, \Gamma g) \mapsto \Gamma g a_t$ is called the \emph{frame flow}.  The following theorem is crucial.
\begin{theorem}[\cite{Winter}]
Let $\Gamma$ be a discrete non-elementary subgroup of $G$. Assume that $\Gamma$ is Zariski-dense and has finite BMS measure. Then the BMS measure on $\Gamma \backslash G$ is (strongly) mixing under the right operation of $A$.
\end{theorem}

We now state some easy facts and do routine checks.

The space $X=\Gamma \backslash G$ is a standard Borel space on which $N$ operates (on the right) in a Borel way with discrete stabilizers. We may disintegrate the Bowen-Margulis-Sullivan measure on $\Gamma \backslash G$ along $N$. Assume now that $\Gamma$ is Zariski-dense and has finite BMS measure. We obtain a measurable mapping $\sigma : X \to \Mrad(N)$. Also, let $a$ be some non-identity element of $A$. There is a real number $t \neq 0$ such that for any $x \in X$ and $n \in N$, 
\[ x n a = xa (a^{-1} n a) = xa h_{e^t}(n) \]
where $h_{e^t}$ is the Heisenberg similitude with ratio $e^t$. We may assume that $t<0$, \emph{i.e.} $h_{e^t}$ is a contracting similitude of $N$ (endowed with the Heisenberg metric).

Last, we know that the BMS measure $\lambda$ on $X$ is $a$-ergodic because of the previous theorem. The conditions stated at the beginning of section \ref{section.dimension} are thus satisfied, and we may consider  the dimension of $\lambda$ along $N$ and $Z$, $\Dim(\lambda,N)$ and $\Dim(\lambda,Z)$ respectively, as well as the transverse dimension $\DimT(\lambda,N/Z)$. One easily checks that
\[ \Dim(\lambda,N) \in [0,2n],\quad \Dim(\lambda,Z) \in [0,2],\quad \DimT(\lambda,N/Z) \in [0,2(n-1)] \text. \]
Indeed, recall that by definition $\Dim(\lambda,N)$ (for example) is the dimension of the conditional measures of $\lambda$ along $N$. Now $N$ has Hausdorff dimension $2n$ (with respect to the Heisenberg metric) so we see that $\Dim(\lambda,N) \in [0,2n]$. In the same way, $Z$ has Hausdorff dimension $2$ (with respect to the restricted Heisenberg metric), and $N/Z$ has Hausdorff dimension $2(n-1)$, indeed it is isometric to $\R^{2(n-1)}$.

We now apply the Ledrappier-Young formula. This yields 
\begin{equation} \label{eq.LY.BMS}
 \Dim(\lambda,N) = \Dim(\lambda,Z)+\DimT(\lambda,N/Z)
\end{equation}

Let us check the hypotheses 1 and 2 (see beginning of section \ref{sect.LY}). First, the right operation of $N$ on $\Gamma \backslash G$ may not have uniformly discrete stabilizers, but the workaround is easy, as this operation is in fact essentially free, \emph{i.e.} the stabilizer of $\lambda$-almost every point is trivial. Indeed, if $g \in G$ is such that $gn=\gamma g$ with $n \in N$ different from the identity and $\gamma \in \Gamma$, then $g^+$ must be a parabolic limit point. There are only countably many parabolic limit points, and the Patterson-Sullivan measure is atomless (because $\Gamma$ has finite BMS measure), so the essential freeness of the operation of $N$ on $\Gamma \backslash G$ is a consequence of the very definition of BMS measure, equation \ref{eq.BMS}. Second, the quotient metric space $N/Z$ (endowed with the quotient of Heisenberg metric) satisfies Besicovitch covering property because it is isometric to euclidean space $\R^{2(n-1)}$. 

We state for future reference the following
\begin{lemma} \label{lemma14}
Keep the previous assumptions and notations. For any $g \in G$, denote by $\phi_g$ the mapping $N \to \partial \Hc^n \setminus \{ g^-\}$, 
\[ \phi_g(n) = (gn)^+ \text. \]
Then for $\lambda$-almost every $x=\Gamma g \in \Gamma \backslash G$, the push-forward $(\phi_g)_* \sigma(x)$ is equivalent to the Patterson-Sullivan measure (restricted to the complement of $g^+$), and the Radon-Nikodym derivative is a continuous mapping 
\[ \partial \Hc^n \setminus \{ g^+ \} \to ]0,+\infty[ \text. \]
\end{lemma}
We skip the straightforward proof.

\section{Hausdorff dimension of limit sets}
We keep the notations and definitions of the previous section.
\subsection{A lower bound for Hausdorff dimension}
Let $\Gamma$ be a non-elementary discrete subgroup of $G=\mathbf{PU}(1,n)$ ($n \geq 2$). If $\Lambda_\Gamma^c$ is the set of conical limit points of $\Gamma$, the Hausdorff dimension of $\Lambda_\Gamma^c$, with respect to the Gromov metric on the boundary at infinity, is equal to $\delta_\Gamma$. Therefore, if we denote by $\mathrm{dim}_E(\Lambda_\Gamma^c)$  the Hausdorff dimension of $\Lambda_\Gamma^c$, with respect to the spherical metric, the following inequalities hold:
\[ \frac{\delta_\Gamma}{2} \leq \mathrm{dim}_E(\Lambda_\Gamma^c) \leq \delta_\Gamma  \]
(lemma \ref{lemma.gromov.spherical}). In fact, by virtue of theorem \ref{th.balogh}, we know that 
\[ \max\left\{\frac{\delta_\Gamma}{2},\delta_\Gamma-1\right\} \leq \mathrm{dim}_E(\Lambda_\Gamma^c) \text.\]

Our aim is to make this result more precise  under some mild assumptions. Let us sketch our argument again. There is a natural fibration on $\mathcal H_n$, namely the centre $Z$ gives rise to a mapping
\[ \mathcal H_n \to \mathcal H_n /Z \text. \]
Now the main difference between Heisenberg metric and euclidean metric on $\mathcal H_n$ is ``on the fibers'', because the quotient metrics on $\mathcal H_n/Z$ are actually the same. Fiberwise, on the other hand, we have
\[ d_E(h,h')=d_H(h,h')^2 \quad(h,h' \text{ belong to the same fiber}) \text. \]

What is left is to understand how the dimension along fibers and the transversal dimension account for the dimension of Patterson-Sullivan measure itself. In general there is no exact relation but thanks to our ``Ledrappier-Young'' formula here we know that dimension is indeed additive as far as Heisenberg metric is concerned. On the other hand when Heisenberg space is endowed with euclidean metric, dimension is only super-additive -- which is why we only get an inequality in the end.

We now state and prove our theorem. 
\begin{theorem} \label{th.lower.bound}
Let $\Gamma$ be a non-elementary discrete subgroup of $G=\mathbf{PU}(1,n)$, Zariski-dense, with finite BMS measure. 
If $\mu$ is some Patterson-Sullivan measure of exponent $\delta_\Gamma$, then for $\mu$-almost every $\xi$, 
\[ \delta_\Gamma - \frac{1}{2} \Dim(\lambda,Z)  \leq \diminf(\mu,\xi) \leq \delta_\Gamma  \]
where the dimension is with respect to the spherical metric on the boundary.

The same inequality holds if we replace $\diminf(\mu,\xi)$ with $\mathrm{dim}_E(\Lambda_\Gamma^c)$.
\end{theorem}

\begin{proof}[Proof of the theorem]
The theorem is a direct consequence of the following two facts:
\begin{itemize}
\item[A.] $\displaystyle{\delta_\Gamma = \Dim(\lambda,Z)+\DimT(\lambda,N/Z)}$
\item[B.] $\displaystyle{\diminf(\mu,\xi) \geq \frac{1}{2} \Dim(\lambda,Z)+\DimT(\lambda,N/Z)}$ for $\mu$-almost every $\xi \in \partial \Hc^n$
\end{itemize}
which we now prove in order.

To prove fact A, we just need to recall equation \ref{eq.LY.BMS} and prove the relation $\Dim(\lambda,N)=\delta_\Gamma$. This holds because the Patterson-Sullivan measure has pointwise lower dimension $\delta_\Gamma$ almost everywhere (with respect to the Gromov metric), see \cite{Ledrappier} theorem 4.3, and because of lemmas \ref{lemma14} and \ref{lemma.loclip}.1.

We now prove fact B. Fix some compact neighbourhood $B$ of the identity in $N$. For any $x \in X$, let
\[ \nu^B(x)=\frac{\sigma(x)|B}{\sigma(x)(B)} \]
and let $\theta^B(x)$ be the pushforward of $\nu^B(x)$ through the quotient mapping $\pi: N \to N/Z$. Let us disintegrate $\nu^B(x)$ above $\theta^B(x)$:
\[ \nu^B(x) = \int \mathrm{d}(\theta^B(x))(v)\ \phi_v^B(x) \]
where $\phi_v^B(x)$ is (almost surely) a Radon measure concentrated on the fiber $\pi^{-1}(v)$.

We know that for $\lambda$-almost every $x \in X$, $\theta^B(x)$ has lower pointwise dimension equal to $\DimT(\lambda,N/Z)$ almost everywhere, and $\phi_v^B(x)$ has exact dimension equal to $\Dim(\lambda,Z)$, for $\theta^B(x)$-almost every $v \in N/Z$. Recall that $N/Z$ is endowed with the quotient metric derived from Heisenberg metric, and this metric coincides with the quotient metric derived from euclidean metric. Also, fibers of the form $\pi^{-1}(v)$ are endowed with the restriction of Heisenberg metric.

Now we endow $N$ with euclidean metric, and each fiber of the form $\pi^{-1}(v)$ is endowed with the restriction of euclidean metric. According to lemma \ref{lemma12}.2, the conditional measure $\pi_v^B(x)$ is of exact dimension $\dfrac{\Dim(\lambda,Z)}{2}$ for $\lambda$-almost every $x$ and $\theta^B(x)$-almost every $v$. According to proposition \ref{prop.additivity.1} we deduce that the lower pointwise dimension of $\nu^B(x)$, with respect to the Euclidean metric, is (almost everywhere) at least equal to
\begin{equation} \DimT(\lambda,N/Z)+\frac{1}{2} \Dim(\lambda,Z)\text. \label{eq.formula.prve}
\end{equation}
By considering an increasing sequence of compact neighbourhood of identity $B_1 \subset B_2 \subset \ldots$ and using lemma \ref{lemma.loclip}.2, we see that the pointwise lower dimension of Patterson-Sullivan measure is almost everywhere greater than the previous number. 
\end{proof}

In fact, we can show that the lower inequality must be strict unless $\Gamma$ is a lattice. I am grateful to the referee for pointing out and correcting a mistake in the previous version of this proposition.
\begin{proposition}
Let $\Gamma$ be a non-elementary discrete subgroup of $G=\mathbf{PU}(1,n)$, Zariski-dense, with finite BMS measure, such that $\Lambda_\Gamma \neq \partial \Hc^n$. Then $\dim(\lambda,Z)<2$. In particular, we get the strict inequality
\[ \delta_\Gamma - 1 < \mathrm{dim}_E(\Lambda_\Gamma^c) \text. \]
\end{proposition}
\begin{proof}
Let us describe briefly the rationale behind this proposition. If $\dim(\lambda,Z)=2$, then it is a classical fact that $\BMS$ must in fact be $Z$-invariant. This implies that the limit set mut be in some way saturated with respect to chains (recall that the boundary of a complex geodesic is called a chain). Since $\Gamma$ is Zariski-dense, its limit set $\Lambda_\Gamma$ may not be included in a chain, so it must really be a ``big'' set, and we show in fact that it must be the whole boundary.

Now we assume that $\dim(\lambda,Z)=2$. Then it is well-known that for $\BMS$-almost every $x \in X$ (recall that $X=\Gamma \backslash G$), $\sigma_Z(x)$ must be the Haar measure on $Z$, and $\sigma(x)$ must then be $Z$-invariant. This key fact is kind of folklore; it is proven (in slightly different languages) in \cite{MargulisTomanov}, as well as in \cite{Hoch} and \cite{EinsiedlerLindenstrauss}.

This means that for almost every $\xi$ on the boundary (with respect to the Patterson-Sullivan measure), if we send $\xi$ to infinity and identify $\partial \Hc^n \setminus \{\xi\}$ with the Heisenberg space $\HH_n$, then $\Lambda_\Gamma \setminus \{ \xi \}$ is $Z$-invariant when seen as a subset of $\HH_n$. Here we are using the fact that $\Lambda_\Gamma$ is the support of the Patterson-Sullivan measure as well as lemma \ref{lemma14}.

Let us reformulate this: for almost every $\xi$ with respect to the Patterson-Sullivan measure, and for \emph{every} $\eta \in \Lambda_\Gamma \setminus \{\xi\}$, the unique chain passing through $\xi$ and $\eta$ is contained in $\Lambda_\Gamma$.

Clearly this implies that there is some chain $C$ that is contained in the limit set $\Lambda_\Gamma$. Since $\Gamma$ is Zariski-dense, $\Lambda_\Gamma$ cannot be equal to $C$; so the Patterson-Sullivan measure of $\Lambda_\Gamma \setminus C$ is $>0$ (as $\Lambda_\Gamma$ is the support of the Patterson-Sullivan measure and $C$ is a closed subset of $\Lambda_\Gamma$).

Now we use the above reformulation again in order to pick  some point $\xi$ not in $C$ such that for \emph{any} $\eta \in C$, the chain passing through $\xi$ and $\eta$ is contained in $\Lambda_\Gamma$. 

Let us send $\xi$ at infinity and look at what is going on in the Heisenberg space $\HH_n$. The subset $L=\Lambda_\Gamma \setminus \{\xi\} \subset \HH_n$ satisfies the following properties:
\begin{itemize}
\item It is vertically saturated; \emph{i.e.} for any $h\in L$, the translate $hZ$ is contained in $L$.
\item There is  some fixed vertical chain $C=h_0 L$ and some point $h_1$ not belonging to $C$ such that every chain passing through $h_1$ and $C$ is included in $L$.
\end{itemize}

We finish the proof assuming that $n=2$; in the general setting, one would have to argue by induction on the dimension of the smallest $k$-chain contained in the limit set $\Lambda_\Gamma$ (a $k$-chain is the boundary of a $k$-complex geodesic).

Let us check that $L=\HH_2$; we identify $\HH_2$ with $\C \times \R$. In order to simplify notations, we assume that the vertical chain $Z$ is included in $L$, and also that the point $(1,0) \in \C \times \R$ belongs to $L$.  
\begin{description}
\item[Claim.] Fix some real number $y$. The circle with centre $1/2+iy$ and radius $\sqrt{1/4+y^2}$ is contained in the vertical projection of $L$ onto $\C$.
\end{description}
Define
\[ v_0= \frac{1}{2}+i \qquad s_0=-2y\qquad r_0=\sqrt{\frac{1}{4}+y^2} = \| v_0 \| \text. \]
For any $\theta \in [0,2\pi]$, let 
\[ v(\theta)=v_0+r_0 e^{i \theta}\qquad s(\theta)=s_0 - 2 \mathrm{Im}( \overline{v_0} v(\theta)) \text. \]
Then the mapping $\theta \mapsto (v(\theta),s(\theta))$ parametrizes the unique chain passing through $(0,s_0)$ and $(1,0)$. Of course the non-trivial assertion is the fact that this mapping does indeed parametrize a chain, see \cite{Goldman} equation (4.12) page 129. This proves the claim, and we deduce that $L=\partial \Hc^2 \setminus \{\xi\}$.

We have shown that $\dim(\lambda,Z)=2$ implies $\Lambda_\Gamma = \partial \Hc^n$; hence the conclusion.
\end{proof}

\begin{corollary} \label{cor.referee}
Let $\Gamma$ be a non-elementary discrete group of $G$, Zariski-dense and geometrically finite; assume furthermore that $\Gamma$ is \emph{not} a lattice.  Then $\dim(\lambda,Z)<2$. In particular, we get the strict inequality
\[ \delta_\Gamma - 1 < \mathrm{dim}_E(\Lambda_\Gamma) \text. \]
\end{corollary}
\begin{proof}
  Recall that geometrical finiteness of $\Gamma$ implies that the limit set $\Lambda_\Gamma$ is the union of the conical limit set $\Lambda_\Gamma^c$ and the parabolic limit set $\Lambda_\Gamma^p$, the latter being countable (hence $\Lambda_\Gamma$ and $\Lambda_\Gamma^c$ have the same Hausdorff dimension); also, the Hausdorff dimension of the limit set $\Lambda_\Gamma$ (with respect to the Gromov metric) is equal to the critical exponent $\delta_\Gamma$.

  Now, by the previous proposition, we know that if $\dim(\lambda,Z)=2$, the limit set must be equal to the boundary $\partial \Hc^n$; thus $\delta_\Gamma$ is equal to $2n-1$ and this implies that $\Gamma$ (which is geometrically finite) is a lattice.
\end{proof}

Before ending this paragraph, let us remark that in view of the main result in \cite{LedXie}, it seems likely that $\dim(\lambda,Z)$ should be strictly less than $\delta_\Gamma$ when $\Gamma$ is Zariski-dense (and has finite BMS measure) but I have not been able to adapt the methods of Ledrappier and Xie to prove this result.

\subsection{Schottky subgroups in good position}
We will now describe a class of Schottky subgroups of $G=\mathbf{PU}(1,n)$ for which, with the notations of theorem \ref{th.lower.bound}, $\dim(\lambda,Z)=0$. As an immediate corollary, the dimension of the limit set associated with such a Schottky subgroup, with respect to the spherical metric, is equal to $\delta_\Gamma$. Recall that a Schottky subgroup is convex cocompact, so that any limit point is a conical limit point.

Let $W \subset G$ be a finite set of hyperbolic transformations, at least two, and, for each $w \in W\cup W^{-1}$, let $B(w)$ be an open subset of $\partial \Hc^n$. We make the following assumptions:
\begin{enumerate}
\item If $w \in W$, then $w^{-1}$ does not belong to $W$.
\item The closures $\overline{B(w)}$ are pairwise disjoint.
\item For any $w \in W \cup W^{-1}$, 
\[ w \left( \partial \Hc^n \setminus B(w^{-1})\right) \subset B(w) \text. \]
\item No chain passes through three of these open subsets $B(w)$.
\end{enumerate}
Recall that a chain is the boundary of a complex geodesic.

It is easy to construct a set $W$ satisfying these hypotheses. First, choose some hyperbolic isometries $w_1',\ldots,w_k'$ such that if $i \neq j$, the chain passing through the fixed points of $w_i'$ does not pass through a fixed point of $w_j'$. Then let $w_i = (w_i')^n$ ($1 \leq i \leq k$), and the set $W=\{w_1,\ldots,w_k\}$ suits our needs, if only $n$ is great enough.

Now let $\Gamma$ be the subgroup of $G$ generated by $W$. It is easy to check that $\Gamma$ is the free group of basis $W$, $F(W)$. Therefore for any element $\gamma \in \Gamma$, we may speak of the reduced decomposition and the length of $\gamma$ (with respect to $W$). We say that $\Gamma$ is a Schottky subgroup \emph{in good position}.

If $f \in \Gamma$ has reduced decomposition $f=f_1 \cdots f_p$ (\emph{i.e.} $f_i f_{i+1} \neq e$ for $i < p$ and $f_i \in W \cup W^{-1}$ for $i \leq p$), we denote by $B(f)$ the set
\[ f_1 \cdots f_{n-1} B(f_n) \text. \]
If $f,g \in \Gamma$ are distinct elements (not equal to identity), the sets $B(f)$ and $B(g)$ have empty intersection. 

If $f,g \in \Gamma$ have reduced decomposition $f_1 \cdots f_p$, $g_1 \cdots g_q$ respectively, we denote by $f \wedge g$ the longest word $h$ with reduced decomposition $h_1 \cdots h_r $ such that $f_i = g_i = h_i$ for $1 \leq i \leq r$. If $f \wedge g$ is the empty word, we say that $f$ and $g$ are disjoint words; on the other hand, if $f \wedge g = f$, we say that $f$ is a prefix of $g$. If $f$ is a prefix of $g$, then $B(g) \subset B(f)$.

\begin{lemma}
Let $f,g,h \in \Gamma$ be such that none of them is a prefix of an other. Then the sets $B(f),B(g),B(h)$ are pairwise disjoints, and no chain passes through these three sets.
\end{lemma}
\begin{proof}
We argue by contradiction. Assume $C$ is a chain passing through $B(f),B(g)$ and $B(h)$. We may assume that $f \wedge g \wedge h$ is the empty word: if not, we replace $C$ with $(f \wedge g \wedge h)^{-1} C$, $f$ with $(f \wedge g \wedge h)^{-1} f$, \emph{etc}. 

If $f,g,h$ are pairwise disjoints, the conclusion is a direct consequence of hypothesis 4 above.

Now assume for example that $k = f \wedge g$ is not the empty word. Note that $k$ is not equal to $f,g$ or $h$ because none of these words is a prefix of an other. Then $k^{-1} C$ is a chain passing through $B(k^{-1} f)$, $B(k^{-1}g)$ and $B(k^{-1} f)$. If $l(w)$ is the length of the element $w \in \Gamma$, we have 
\[ l(k^{-1}f)+l(k^{-1}g)+l(k^{-1}h) < l(f)+l(g)+l(h) \]
and we may argue by induction on $l(f)+l(g)+l(h)$ to prove the lemma.
\end{proof}

\begin{lemma}
Any chain contains at most two points of $\Lambda_\Gamma$.
\end{lemma}
\begin{proof}
Since $\Gamma$ is free, for any $\xi \in \Lambda_\Gamma$ and any integer $n \geq 1$, there is a unique word $\xi(n) \in \Gamma$ of length $n$ such that $\xi$ belongs to $B(\xi(n))$. We let $\xi(n)=\xi_1 \cdots \xi_n$, where $\xi_i$ belongs to $W \cup W^{-1}$ ($1 \leq i \leq n$) and $\xi_i \xi_{i+1} \neq e$.

Now assume on the contrary that $\xi,\eta,\zeta$ are three distinct points of $\Lambda_\Gamma$ that belong to some chain $C$. Since $\xi,\eta,\zeta$ are different, we may find an index $n$ such that the three words
\[ f=\xi_1 \cdots \xi_n,\quad g=\eta_1 \cdots \eta_n,\quad h=\zeta_1 \cdots \zeta_n \]
are pairwise different. By virtue of the previous lemma, no chain passes through $B(f),B(g)$ and $B(h)$. This contradicts the fact that $C$ does pass through these sets.
\end{proof}

We are now able to prove the following fact.
\begin{proposition}
Let $\Gamma$ be a Schottky subgroup in good position of $\mathbf{PU}(1,n)$ ($n \geq 2$). Let $\lambda$ be the BMS measure on $\Gamma \backslash G$ and recall that $Z$ is the centre of $N$. Then
\[ \Dim(\lambda,Z)=0 \text. \]
\end{proposition}
\begin{proof}
Let $\sigma : X \to \Mrad(N)$ be the mapping obtained by disintegrating $\lambda$ along $N$ on the right. For $\lambda$-almost every $x=\Gamma g$, the support of $\sigma(x)$ is the inverse image, by the mapping $\phi : N \to \partial \Hc^n \setminus \{g^-\}$, $n \mapsto (gn)^+$, of the set $\Lambda_\Gamma \setminus \{g^-\}$. For any $n \in N$, the translate $nZ$ is the inverse image, by $\phi$, of the chain passing through $(gn)^+$ and $g^-$.

Now, for $\lambda$-almost every $x = \Gamma g$, and $\sigma(x)$-almost every $n \in N$, both points $g^-$ and $(gn)^+$ belong to the limit set. Consequently, the chain passing through these points does not contain any other point of $\Lambda_\Gamma$, so that $nZ$ meets the support of $\sigma(x)$ only at $n$. 

In other words, $\sigma(x)$ is concentrated on a Borel section of the quotient mapping $N \to N/Z$. Therefore, the conditional of $\lambda$ along $Z$ are Dirac measures, and in particular $\Dim(\lambda,Z)=0$.
\end{proof}

\begin{corollary} \label{cor.schottky}
The Hausdorff dimension of $\Lambda_\Gamma$, with respect to the spherical metric, is equal to $\delta_\Gamma$.
\end{corollary}

Before ending this paragraph, let us remark that being in good position is not a \emph{generic} condition. Indeed it is clear by definition that a Schottky subgroup in bad position will remain so after a small perturbation, because in the definition the sets $B(w)$ are taken to be open. On the other hand, being in good position is an \emph{open} condition since we may replace the sets $B(w)$ with their closure (maybe after shrinking them a little bit).

\subsection{Discussion}
First, let us discuss our hypotheses. 

Zariski density is absolutely crucial. Indeed, let us say, for example, that $n=2$. Pick some one-dimensional complex geodesic $H \subset \Hc^2$. There is a bijective mapping $\phi : H \to \mathbf{H}_{\mathbf{R}}^1$ onto the real hyperbolic line, such that
\[ d_{\mathbf{H}_{\mathbf{R}}^1}(\phi(\xi),\phi(\eta))=\frac{1}{2} d_{\Hc^2}(\xi,\eta) \]
(where the letter $d$ denotes hyperbolic metric on both spaces). As a consequence, if $\Gamma$ is some uniform lattice of $\mathbf{PSL}_2(\mathbf{C})$, we may identify $\Gamma$ with some discrete subgroup of $ \mathbf{PU}(1,2)$ with finite Bowen-Margulis-Sullivan measure, such that $\delta_\Gamma=2$ and $\Lambda_\Gamma=H$. The Hausdorff dimension of $\Lambda_\Gamma$ with respect to the \emph{spherical} metric is then equal to $1$. Of course, such a $\Gamma$ is not Zariski-dense (it is contained in some conjugate of $\mathbf{PU}(1,1)$).

On the other hand, it is not clear to the author wether  finiteness of the Bowen-Margulis-Sullivan measure is a mere technical hypothesis. It would be interesting to find out what can happen if we assume only divergence of the Poincaré series.

Now we discuss our results and ask a few questions. 

As was already mentioned in the introduction to this paper, corollary \ref{cor.schottky} has an interesting consequence: if $\Gamma$ is a well-positionned Schottky subgroup of $\mathbf{PU}(1,n)$, one must have
\[ \delta_\Gamma \leq 2(n-1) \text. \]
This raises the following 
\begin{questions}
Let $\Gamma$ be some Zariski-dense Schottky subgroup of $\mathbf{PU}(1,n)$. 
\begin{enumerate}
\item Does the previous inequality hold?
\item Is the Hausdorff dimension of $\Lambda_\Gamma$ with respect to the \emph{spherical} metric equal to $\delta_\Gamma$?
\item Does  the dimension of the Bowen-Margulis-Sullivan measure along the centre $Z$ have to be zero?
\end{enumerate}
\end{questions}
Recall that $Z$ is the centre of the Heisenberg group. Of course, 3 implies 2 and 2 implies 1. I would like to underline the fact that these questions are genuine, and I have absolutely no heuristics whatsoever that would imply that any of these assertions should be considered plausible.

It is worth noting that another family of subsets of the Heisenberg space that exhibit coincidence of dimension with respect to both Euclidean and Heisenberg metrics was constructed by Balogh \emph{et al} in \cite{BaloghTyson} and \cite{BaloghLifts}.

Now let us pass to general,  non necessarily Schottky subgroups. Let $\Gamma$ be a Zariski dense discrete subgroup of $\mathbf{PU}(1,n)$ with finite Bowen-Margulis-Sullivan measure. Theorem \ref{th.lower.bound} makes it plain that in order to understand the Hausdorff dimension (with respect to the spherical metric) of $\Lambda_\Gamma$, one would like to be able to compute the dimension of the Bowen-Margulis-Sullivan measure along the centre $Z$, $\Dim(\lambda,Z)$. Because of the Ledrappier-Young formula, this is equivalent to computing the dimension of the Bowen-Margulis-Sullivan measure transverse to the centre $Z$, $\DimT(\lambda,N/Z)$. Indeed, recall that the Ledrappier-Young formula yields
\[ \delta_\Gamma = \Dim(\lambda,Z)+\DimT(\lambda,N/Z) \text. \]

\begin{questions}
Let $\Gamma$ be some Zariski dense discrete subgroup of $\mathbf{PU}(1,n)$ with finite Bowen-Margulis-Sullivan measure. Let $\delta_\Gamma$ be the growth exponent of $\Gamma$.
\begin{enumerate}
\item Is it true that if $\delta_\Gamma \leq 2(n-1)$, then $\Dim(\lambda,Z)=0$?
\item Is it true that if $\delta_\Gamma > 2(n-1)$, then $\DimT(\lambda,N/Z)=2(n-1)$?
\end{enumerate}
\end{questions}
If the answer to both questions was affirmative, this would, by virtue of theorem \ref{th.lower.bound}, and thanks to the formula of Balogh \emph{et al}, imply the following formula, where we denote by $\dim_E(\Lambda_\Gamma^c)$ the Hausdorff dimension of the conical limit set with respect to the spherical metric:
\begin{equation} \tag{*} \dim_E(\Lambda_\Gamma^c) = \left \{ \begin{array}{ccc} \delta_\Gamma & \text{if} & \delta_\Gamma \leq 2(n-1) \\ 2(n-1)+\frac{1}{2} (\delta_\Gamma-2(n-1)) & \text{if} & \delta_\Gamma > 2(n-1)
    \end{array} \right.
\end{equation}

What this (hypothetical) formula means is that the Hausdorff dimension of the (conical) limit set with respect to the spherical metric is really as high as it can be. Geometrically, this would imply in some way that the limit set is as nearly transverse to the centre $Z$ as possible.

To put differently, we are asking wether  there is a ``transverse measure saturation'' phenomenon going on; to be more specific, does it hold that when the growth exponent is ``small'' (\emph{i.e.} $\delta_\Gamma \leq 2(n-1)$) the transverse measure accounts for all of the dimension (\emph{i.e.} the dimension along $Z$ is zero), whereas as the growth exponent gets larger, the transverse measure is in some way ``saturated'' (perhaps absolutely continuous with respect to the Lebesgue measure on the quotient space $N/Z$) and the measure along $Z$ is forced to take its share of the dimension?

In pondering these questions, one should bear in mind related examples:
\begin{itemize}
\item In the Euclidean $\R^n$ space, if $\mu$ is some exact-dimensional (say) probability measure of dimension $\delta$, and if we randomly pick some $k$-plane $P \subset \R^n$ (with respect to the standard Lebesgue measure on the space of $k$-planes in $\R^n$), then the orthogonal projection of $\mu$ onto $P$ has dimension
\[ \inf\{ \delta,k \} \]
and furthermore if we disintegrate $\mu$ above $P$ (along the orthogonal projection), the conditional measures have dimension $0$ if $\delta \leq k$ and $\delta - k$ otherwise. This is basically Marstrand's theorem, see \cite{Mattila} or \cite{thesis}.
\item Falconer's formula giving the almost sure Hausdorff dimension of self-affine subsets of the Euclidean space \cite{Falconer2}. If we pretend that when the boundary at infinity is endowed with the spherical metric, the elements of $\Gamma$ are affine transformations of the euclidean space $\R^{2n-2}$, and try and apply Falconer's formula, in order to compute the Hausdorff dimension of the (conical) limit set with respect to the spherical metric, we do end up with the above formula $(*)$.
\item In \cite{LedLin}, Ledrappier and Lindenstrauss exhibit a related ``transverse measure saturation'' phenomenon (though for some reason in that setting the phenomenon holds only in small dimension).
\end{itemize}
I reckon that I have no idea wether there are actually non-trivial (\emph{i.e.} non lattices) $\Gamma$ that satisfy $\delta_\Gamma > 2(n-1)$. In the quaternionic hyperbolic setting, it is known that there is in fact a gap, see \cite{Corlette}.

\begin{question}
Can we find a Zariski dense discrete subgroup $\Gamma$ of $\mathbf{PU}(1,n)$ with finite Bowen-Margulis-Sullivan measure, that is not a lattice, and such that
\[ \dim(\lambda,Z)>0 \quad \text{?} \]
\end{question}

Let us note as well the following  question.
\begin{question}
Let $\Gamma$ be some Zariski dense discrete subgroup of $\mathbf{PU}(1,n)$ with finite Bowen-Margulis-Sullivan measure. Is the Hausdorff dimension of the conical limit set with respect to the spherical measure on the boundary equal to
\[ \delta_\Gamma - \frac{1}{2} \dim(\lambda,Z) \quad \text{?}\]
\end{question}
This is weaker than question 3 above.

Finally, I would like to mention that in the real hyperbolic setting, I am able to compute the dimension of the Bowen-Margulis-Sullivan measure along (connected) subgroups of the unipotent group $N$ (where $\mathbf{PO}(1,n)=KAN$ is an Iwasawa decomposition). This stands in stark contrast to the complex hyperbolic setting we have been looking at in this paper, where computing $\dim(\lambda,Z)$ seems more difficult. The ``transverse measure saturation'' phenomenon we mentioned earlier does take place in the real hyperbolic setting. In fact this is essentially a consequence of Marstrand's theorem. This problem is related to a paper of Oh and Mohammadi, \cite{OhMohammadi}. The reader is referred to \cite{preprintOM}.

\bibliography{biblio}

\end{document}